\documentclass[11pt]{amsart}
\title[minimal models of solvmanifolds with local systems]
{minimal models, formality and hard Lefschetz properties of solvmanifolds with local systems }

\author{Hisashi Kasuya}
\usepackage{amssymb}
\usepackage{amsmath}
\usepackage{amscd}
\usepackage{amstext}
\usepackage{amsfonts}
\usepackage[all]{xy}

\address[H.kasuya]{Graduate school of mathematical science \\
  University of tokyo\\
3-8-1 Komaba Meguro-ku Tokyo 153-8914\\
   Japan
}
\curraddr{}
\email{khsc@ms.u-tokyo.ac.jp}

\thanks{}
\keywords{solvmanifold, Sullivan's minimal model, formality, Hard Lefschetz property, flat connection, Higgs bundle, polycyclic group, algebraic hull}
\newcommand{\C}{\mathbb{C}}
\newcommand{\R}{\mathbb{R}}

\newcommand{\Q}{\mathbb{Q}}
\newcommand{\Z}{\mathbb{Z}}
\newcommand{\g}{\frak{g}}
\newcommand{\n}{\frak{n}}

\newcommand{\Aut}{\rm Aut}

\theoremstyle{plain}
\newtheorem{theorem}{Theorem}[section]
\theoremstyle{plain}
\newtheorem{construction}{Construction}[section]
\theoremstyle{plain}
\newtheorem{remark}{Remark}[section]
\theoremstyle{lemma}
\newtheorem{lemma}[theorem]{Lemma}
\theoremstyle{definition}

\theoremstyle{proposition}
\newtheorem{proposition}[theorem]{Proposition}
\theoremstyle{corollary}
\newtheorem{corollary}[theorem]{Corollary}
\theoremstyle{remark}
\newtheorem{example}{\bf Example}
\begin{document} 
\maketitle
\begin{abstract}
For a simply connected  solvable Lie group $G$ with a cocompact discrete subgroup $\Gamma $, we consider the space of differential forms on the solvmanifold $G/\Gamma$ with values  in certain flat bundle so that this space has a structure of a differential graded algebra(DGA).
We construct   Sullivan's minimal model of this  DGA.
This result is an extension of Nomizu's theorem  for ordinary coefficients in the nilpotent case.
By using this result, we  refine Hasegawa's result of formality of nilmanifolds  and Benson-Gordon's result of hard Lefschetz properties of nilmanifolds.
\end{abstract}
\tableofcontents
\section{Introduction}
The main purpose of this paper is to compute the de Rham cohomology of solvmanifolds with values in local coefficients associated to some diagonal representations by using of the invariant forms and the unipotent hulls.
The computations are natural extensions of Nomizu's computations of untwisted de Rham cohomology of nilmanifolds by the invariant forms in \cite{Nom}.
The computations give natural extensions of Hasegawa's result of formality of nilmanifolds (\cite{H}) and Benson and Gordon's result of hard Lefschetz properties of nilmanifolds (\cite{BG}).

First we explain the central tools of this paper called the unipotent hulls and algebraic hulls.
Let $G$ be a a simply connected solvable Lie group, there exists a unique algebraic group ${\bf H}_{G}$ called the algebraic hull of $G$ with an injection $\psi:G\to {\bf H}_{G}$ so that:\\
(1)  \ $\psi (G)$ is Zariski-dense in ${\bf H}_{G}$.\\
(2) \ The centralizer $Z_{{\bf H}_{G}}({\bf U}({\bf H}_{G}))$ of ${\bf U}({\bf H}_{G})$ is contained in  ${\bf U}({\bf H}_{G})$.\\
(3) \ $\dim {\bf U}({\bf H}_{G})=\dim G$.   \\
where we denote ${\bf U}({\bf H})$ the unipotent radical of an algebraic group $\bf H$.
We denote ${\bf U}_{G}={\bf U}({\bf H}_{G})$ and call it the unipotent hull of $G$.

We consider Hain's DGAs in \cite{Hai} which are expected to be effective techniques for studying de Rham homotopy theory of non-nilpotent spaces.
Let $M$ be a $C^{\infty}$-manifold and $\rho :\pi_{1}(M,x)\to (\C^{\ast})^{n}$ a representation and $\bf T$  the Zariski-closure of $\rho(\pi_{1}(M,x))$ in $ (\C^{\ast})^{n}$.
Let $\{V_{\alpha}\}$ be the set of one-dimensional representations for all characters $\alpha$ of $\bf T$ and $(E_{\alpha}, D_{\alpha})$ be a rank one flat bundle with the monodromy $\alpha \circ \rho$ and $A^{\ast}(M,E_{\alpha})$  the space of $E_{\rho}$-valued $C^{\infty}$-differential forms.
Denote $A^{\ast}(M, {\mathcal O}_{\rho})=\bigoplus_{\alpha} A^{\ast}(M,E_{\alpha})$ and $D=\bigoplus_{\alpha} D_{\alpha}$.
Then $(A^{\ast}(M, {\mathcal O}_{\rho}), D)$ is a cohomologically connected(i.e. the 0-th cohomology is isomorphic to the ground field) DGA.
In this paper we construct Sullivan's minimal model(\cite{Sul}) of such DGAs on solvmanifolds.

On simply connected solvable Lie groups, we consider DGAs of  left-invariant differential forms with local systems which are analogues of Hain's DGA's.
Suppose $G$ is a simply connected solvable Lie group and $\g$ is the Lie algebra of $G$.
Consider the adjoint representation ${\rm Ad}:G\to  {\rm Aut}(\g)$ and its derivation ${\rm ad}:\g\to D(\g)$ where $D(\g)$ be the Lie algebra of the derivations of $\g$.
We construct  representations of $\g$ and $G$ as following.
\begin{construction}\label{con}
Let $\n$ be the nilradical of $\g$.
There exists a subvector space (not necessarily Lie algebra) $V$ of $\g$ so that
$\g=V\oplus \n$ as the direct sum of vector spaces and for any  $A,B\in V$ $({\rm ad}_A)_{s}(B)=0$ where $(ad_A)_{s}$  is the semi-simple part of ${\rm ad}_{A}$ (see \cite[Proposition I\hspace{-.1em}I\hspace{-.1em}I.1.1] {DER}).
We define the  map ${\rm ad}_{s}:\g\to D(\g)$ as 
${\rm ad}_{sA+X}=({\rm ad}_{A})_{s}$ for $A\in V$ and $X\in \n$.
Then we have $[{\rm ad}_{s}(\g), {\rm ad}_{s}(\g)]=0$ and ${\rm ad}_{s}$ is linear (see \cite[Proposition I\hspace{-.1em}I\hspace{-.1em}I.1.1] {DER}).
Since we have $[\g,\g]\subset \n$,  the  map ${\rm ad}_{s}:\g\to D(\g)$ is a representation and the image ${\rm ad}_{s}(\g)$ is abelian and consists of semi-simple elements.
We denote by ${\rm Ad}_{s}:G\to {\rm Aut}(\g)$ the extension of ${\rm ad}_{s}$.
Then ${\rm Ad}_{s}(G)$ is diagonalizable.
\end{construction}

Let $\bf T$ be the Zariski-closure of ${\rm Ad}_{s}(G)$ in ${\rm Aut}(\g_{\C})$.
Then $\bf T$ is diagonalizable.
Let $\{ V_{\alpha}\}$ be the set of one-dimensional representations for all characters $\alpha$ of $\bf T$.
We consider $V_{\alpha}$ the representation of $\g$ which is the derivation of $\alpha\circ {\rm Ad}_{s}$.
Then we have the cochain complex of Lie algebra $(\bigwedge \g^{\ast}_{\C}\otimes V_{\alpha}, d_{\alpha})$.
Denote $A^{\ast}(\g_{\C},{\rm ad}_{s})= \bigoplus_{\alpha} \bigwedge \g^{\ast}_{\C}\otimes V_{\alpha}$ and $d=\bigoplus_{\alpha} d_{\alpha}$.
Then $(A^{\ast}(\g_{\C},{\rm ad}_{s}), d)$ is a cohomologically connected  DGA. 
In this paper we compute the cohomology of this DGA by the unipotent hull ${\bf U}_{G}$ of $G$.
Let ${\frak u}$ be the Lie algebra of ${\bf U}_{G}$ and $\bigwedge {\frak u}^{\ast}$ be the cochain complex of the dual space $\frak u^{\ast}$ of $\frak u$. We prove the following theorem.
\begin{theorem}[Theorem \ref{qua}]\label{mmmt} 
We have a quasi-isomorphism (i.e. a morphism which induces a cohomology isomorphism) of DGAs
\[\bigwedge {\frak u}^{\ast} \to A^{\ast}(\g_{\C},{\rm ad}_{s}).
\] 
Thus $\bigwedge {\frak u}^{\ast}$ is Sullivan's minimal model of $A^{\ast}(\g_{\C},{\rm ad}_{s})$.
\end{theorem}
Suppose $G$ has  a lattice $\Gamma$ i.e. a cocompact discrete subgroup of $G$.
We call a compact homogeneous space $G/\Gamma$ a solvmanifold.
We have $\pi_{1}(G/\Gamma)\cong \Gamma$.
For the restriction of the semi-simple part of the adjoint representation ${\rm Ad}_{s\vert_{\Gamma}} $ on $\Gamma$, we consider Hain's DGA $A^{\ast}(G/\Gamma,{\mathcal O}_{{\rm Ad}_{s\vert_{\Gamma}}})$.
By using Theorem \ref{mmmt}, we prove:

\begin{theorem}[Corollary \ref{ISOINV}]\label{GENN1}
Let $G$ be a simply connected solvable Lie group with a lattice $\Gamma$ and ${\bf U}_{G}$ be the unipotent hull of $G$.
Let $\frak u$ be the Lie algebra of ${\bf U}_{G}$.
Then we have a quasi-isomorphism
\[\bigwedge {\frak u}^{\ast}\to A^{\ast}(G/\Gamma,{\mathcal O}_{{\rm Ad}_{s\vert_{\Gamma}}}).
\]
Thus $\bigwedge {\frak u}^{\ast}$ is Sullivan's minimal model of $A^{\ast}(G/\Gamma,{\mathcal O}_{{\rm Ad}_{s\vert_{\Gamma}}})$.
\end{theorem} 
If $G$ is nilpotent, the adjoint operator $\rm Ad$ is a unipotent representation and hence $A^{\ast}(G/\Gamma,{\mathcal O}_{{\rm Ad}_{s\vert_{\Gamma}}})=A^{\ast}_{\C}(G/\Gamma)$ and 
$A^{\ast}(\g_{\C},{\rm ad}_{s})=\bigwedge \g_{\C}=\bigwedge {\frak u}^{\ast}$.
In this case,  Theorem \ref{GENN1} reduce to the classical theorem proved by Nomizu in \cite{Nom}. 
Moreover this result gives more progressed  computations of untwisted de Rham cohomology of solvmanifolds  than the results of Mostow and Hattori (see Corollary \ref{GENN} and Section \ref{APP}).

We call a DGA $A$ formal if there exists a finite diagram of morphisms
\[A\rightarrow C_{1}\leftarrow C_{2}\cdot \cdot \cdot \leftarrow H^{\ast}(A)
\]
such that all morphisms are quasi-isomorphisms and we call manifolds $M$ formal if the de Rham complex $A^{\ast}(M)$ is formal.
In \cite{H} Hasegawa showed that formal nilmanifolds are tori.
By the results of this paper, we have a natural extension of  Hasegawa's theorem for solvmanifolds.
\begin{theorem}[Theorem \ref{SOLFORM}]\label{FO1}
Let $G$ be a simply connected solvable Lie group.
Then the following conditions are equivalent:\\
$(A)$  The DGA $A^{\ast}(\g_{\C},{\rm ad}_{s})$ is formal. \\
$(B)$   ${\bf U}_{G}$ is abelian.\\
$(C)$   $G=\R^{n}\ltimes_{\phi} \R^{m}$ such that the action $\phi:\R^{n}\to {\rm  Aut} (\R^{m})$ is semi-simple.\\
Moreover suppose $G$ has a lattice $\Gamma$.
Then the above three conditions are equivalent to the following condition:\\
$(D)$ \ $A^{\ast}(G/\Gamma,{\mathcal O}_{{\rm Ad}_{s\vert_{\Gamma}}})$ is formal.
\end{theorem}

In \cite{BG} Benson and Gordon showed that symplectic nilmanifolds with the hard Lefschetz properties are tori.
We can also have an extension of Benson and Gordon's theorem.
\begin{theorem}[Theorem \ref{SOLHARD}]\label{Ha1}
Let $G$ be a simply connected solvable Lie group.
Suppose $\dim G=2n$ and $G$ has an $G$-invariant symplectic form $\omega$.
Then the following conditions are equivalent:\\
$(A)$  \[[\omega]^{n-i}\wedge: H^{i}(A^{\ast}(\g_{\C},{\rm ad}_{s}))\to H^{2n-i}(A^{\ast}(\g_{\C},{\rm ad}_{s}))
\]
is an isomorphic for any $i\le n$.\\
$(B)$   ${\bf U}_{G}$ is abelian.\\
$(C)$  $G=\R^{n}\ltimes_{\phi} \R^{m}$ such that the action $\phi:\R^{n}\to {\rm  Aut} (\R^{m})$ is semi-simple.

Suppose $G$ has a lattice $\Gamma$ and $G/\Gamma$ has a symplectic form(not necessarily $G$-invariant)  $\omega$.
Then the  conditions $(B)$ and $(C)$ are equivalent to the following condition:\\
$(D)$ \[[\omega]^{n-i}\wedge: H^{i}(A^{\ast}(G/\Gamma,{\mathcal O}_{{\rm Ad}_{s\vert_{\Gamma}}}))\to H^{2n-i}(A^{\ast}(G/\Gamma,{\mathcal O}_{{\rm Ad}_{s\vert_{\Gamma}}}))
\] 
is an isomorphism  for any $i\le n$ \\
\end{theorem}

\begin{remark}
As a representation in an algebraic group, ${\rm Ad}_{s}$ is independent of the choice of a subvector space $V$ in Construction \ref{con} (see Lemma \ref{inde}).
By this, the structures of DGAs $A^{\ast}(\g_{\C},{\rm ad}_{s})$ and $A^{\ast}(G/\Gamma,{\mathcal O}_{{\rm Ad}_{s\vert_{\Gamma}}})$ are
independent of the choice of a subvector space $V$. 
\end{remark}

Finally we consider relations with K\"ahler geometries.
We review studies of K\"ahler structures on solvmanifolds briefly.
See \cite{BC} and \cite{Hn} for more details.
In \cite{BGs} Benson and Gordon conjectured that for a completely solvable simply connected Lie group $G$ with a lattice $\Gamma$, $G/\Gamma$ has a K\"ahler metric if and only if $G/\Gamma$ is a torus.
In \cite{Hc} Hasegawa  studied K\"ahler structures on some classes of    solvmanifolds which are not only completely solvable type and  suggested a generalized version of Benson-Gordon's conjecture: A compact
solvmanifold can have a Kahler structure if and only if it is a finite quotient of a
complex torus that is a holomorphic fiber bundle over a complex torus with fiber a
complex torus. 
In \cite{Ara} Arapura showed Benson-Gordon's conjecture and also showed that the fundamental group of a K\"ahler solvmanifold is virtually abelian by the result in \cite{AN}.
In \cite{Ara} a proof of Hasegawa's conjecture was also written
but we notice that  this proof contains a gap and Hasegawa complement in \cite{Hn}.
We also notice that Baues and Cort\'es showed a more generalized  version of Benson-Gordon's conjecture  for aspherical manifolds with polycyclic fundamental groups in \cite{BC}.

By the theory of Higgs bundle studied by Simpson, we have a twisted  analogues of formality(see \cite{DGMS}) and the hard Lefschetz properties of compact K\"ahler manifolds.
We have:
\begin{theorem}[Special case of Thoerem \ref{fol}]\label{fol1}
Suppose $M$ is a compact K\"ahler manifold with a K\"ahler form $\omega$ and $\rho:\pi_{1}(M)\to (\C^{\ast})^{n}$ is  a representation.
Then the following conditions hold:\\
{\rm (A)} (formality) The DGA  $A^{\ast}(M,{\mathcal O}_{\rho})$ is formal.\\
{\rm (B)}(hard Lefschetz) For any $0\leq i\leq n$ the  linear operator
\[ [\omega]^{n-i}\wedge: H^{i}(A^{\ast}(M,{\mathcal O}_{\rho})) \to H^{2n-i}(A^{\ast}(M,{\mathcal O}_{\rho})) \]
is an isomorphism where $\dim _{\R} M=2n$.

\end{theorem}

Now by Theorem \ref{fol1} formality and hard Lefschetz property of  DGA $A^{\ast}(G/\Gamma,{\mathcal O}_{{\rm Ad}_{s\vert_{\Gamma}}})$ are criteria for $G/\Gamma$ to has a K\"ahler metric.
We will see such conditions are stronger than untwisted formality and hard Lefschetz property.
\begin{remark}\label{exaaa1}
There exist examples of solvmanifolds $G/\Gamma$ which satisfy formality and the hard Lefschetz property of the untwisted de Rham complex  $A^{\ast}(G/\Gamma,)$ but do not satisfy  formality and the hard Lefschetz property of 
$A^{\ast}(G/\Gamma,{\mathcal O}_{{\rm Ad}_{s\vert_{\Gamma}}})$.
\end{remark}
However we will see these criteria can not classify K\"ahler solvmanifolds completely.
\begin{remark}\label{exaaa2}
There exist examples of non-K\"ahler solvmanifolds which satisfy formality and the hard Lefschetz property of $A^{\ast}(G/\Gamma,{\mathcal O}_{{\rm Ad}_{s\vert_{\Gamma}}}) $.
\end{remark}

\section{Preliminaries on  algebraic hulls}

Let $G$ be a discrete group (resp. a Lie group).
We call a map $\rho:G\to GL_{n}(\C)$ a representation, if $\rho$ is a  homomorphism  of  groups (resp. Lie groups).

\subsection{Algebraic groups}
In this paper an algebraic group means an affine algebraic variety $\bf G$ over $\C$ with a group structure such that the multiplication and inverse are morphisms of varieties.
All algebraic groups in this paper arise as Zariski-closed subgroups of $GL_{n}(\C)$.
Let $k$ be a subfield of $\C$.
We call $\bf G$ $k$-algebraic if $\bf G$ is defined by polynomials with coefficient in $k$.
We denote ${\bf G}(k)$  the $k$-points of ${\bf G}$.
We say that an algebraic group is diagonalizable if it is isomorphic to a closed subgroup of $(\C^{\ast})^{n}$ for some $n$. 

\subsection{Algebraic hulls}
A group $\Gamma$ is polycyclic if it admits a sequence 
\[\Gamma=\Gamma_{0}\supset \Gamma_{1}\supset \cdot \cdot \cdot \supset \Gamma_{k}=\{ e \}\]
of subgroups such that each $\Gamma_{i}$ is normal in $\Gamma_{i-1}$ and $\Gamma_{i-1}/\Gamma_{i}$ is cyclic.
For a polycyclic group $\Gamma$, we denote ${\rm rank}\,\Gamma=\sum_{i=1}^{i=k} {\rm rank}\,\Gamma_{i-1}/\Gamma_{i}$.
Let $G$ be a simply connected solvable Lie group and $\Gamma$ be a lattice in $G$.
Then $\Gamma$ is torsion-free polycyclic and $ \dim G={\rm rank}\,\Gamma$ (see \cite[Proposition 3.7]{R}).
Let $\rho :G\to GL_{n}(\C)$, for $g\in G$ be a representation.
Let $\bf G$ and $\bf G^{\prime}$ be the Zariski-closures of $\rho(G)$ and $\rho(\Gamma)$ in $GL_{n}(\C)$.
Then we have ${\bf U}({\bf G})={\bf U}({\bf G}^{\prime})$ (see \cite[Theorem 3.2]{R}).

We review the algebraic hulls.
\begin{proposition}{\rm (\cite[Proposition 4.40]{R})}\label{ttt}
Let $G$ be a simply connected solvable Lie group (resp. torsion-free 
polycyclic group).
Then there exists a unique $\R$-algebraic group ${\bf H}_{G}$ with an injective group homomorphism $\psi :G\to {\bf H}_{G}(\R)  $ 
so that:
\\
$(1)$  \ $\psi (G)$ is Zariski-dense in $\bf{H}_{G}$.\\
$(2)$ \   $Z_{{\bf H}_{G}}({\bf U}({\bf H}_{G}))\subset {\bf U}({\bf H}_{G})$.\\
$(3)$ \ $\dim {\bf U}({\bf H}_{G})$=${\rm dim}\,G$(resp. ${\rm rank}\, G$).   \\
Such ${\bf H}_{G}$ is called the algebraic hull of $G$.
\end{proposition}
We denote ${\bf U}_{G}={\bf U}({\bf H}_{G})$ and call ${\bf U}_{G}$ the unipotent hull of $G$.

\subsection{Direct constructions of algebraic hulls}

Let $\g$ be a solvable Lie algebra, and $\n=\{X\in \g\vert {\rm ad}_{X}\, \, {\rm is\,\, nilpotent}\}$.
$\n$ is the maximal nilpotent ideal of $\g$ and called the nilradical of $\g$.
 Then we have $[\g, \g]\subset \n$.
Consider the adjoint representation ${\rm ad}:\g\to D(\g)$
and the representation ${\rm ad}_{s}:\g\to D(\g)$ as Construction \ref{con}.

 Let $\bar{\g} ={\rm Im} \,{\rm ad}_{s}\ltimes\g$
and
 \[\bar{\n}=\{X-{\rm ad}_{sX}\in  \bar{\g}  \vert X\in\g\}.\]
Then we have $[\g,\g]\subset \n\subset \bar\n$ and $\bar\n$ is the nilradical of $\bar \g$ (see \cite{DER}).
Hence we have $\bar \g= {\rm Im} \,{\rm ad}_{s}\ltimes \bar{\n}$.

\begin{lemma}\label{semmm}
Suppose $\g=\R^{k}\ltimes _{\phi} \n$ such that $\phi$ is a semi-simple action and $\n$ is nilpotent.
Then $\bar \n=\R^{k}\oplus \n$.
\end{lemma}
 \begin{proof}
By the assumption, for $X+Y\in \R^{k}\ltimes _{\phi} \n$, we have ${\rm ad}_{sX+Y}={\rm ad}_{X}$.
Hence we have 
\[
[X_{1}+Y_{1}-{\rm ad}_{sX_{1}+Y_{1}},X_{2}+Y_{2}-{\rm ad}_{sX_{2}+Y_{2}}]=[X_{2},Y_{2}]
\]
for $X_{1}+Y_{1}, X_{2}+Y_{2}\in \R^{k}\ltimes _{\phi} \n$.
Hence the lemma follows.
\end{proof}

Let $G$ be a simply connected solvable Lie group and $\g$ be the Lie algebra of $G$.
Let $N$ be the subgroup of $G$ which corresponds to the nilradical $\n$ of $\g$.
We consider the exponential map $\exp: \g\to G$.
In general $\exp$ is not a diffeomorphism.
But we have the useful property of $\exp$ as following.
\begin{lemma}{\rm(\cite[Lemma 3.3]{Dek})}\label{exponnn}
Let $V$ be a subvector space (not necessarily Lie algebra) $V$ of $\g$ so that
$\g=V\oplus \n$ as the direct sum of vector spaces.
We define the map $F:\g=V\oplus \n\to G$  as
$F(A+X)=\exp (A) \exp(X)$ for $A\in V$, $X\in \n$.
Then $F$ is a diffeomorphism and we have the commutative diagram
\[\xymatrix{
1 \ar[r]&N\ar[r]& G\ar[r] & G/N\cong \R^{k} \ar[r] &1 \\
0\ar[r]&	\n\ar[r] \ar^{\exp}[u] &\g\ar[r]\ar^{F}[u]&\g/\n\cong \R^{k} \ar_{\exp={\rm id}_{\R^{k}}}[u]\ar[r]&0
 }
\]
where $\dim G/N=k$.
\end{lemma}

By this Lemma, for $A\in V$, $X\in \n$, the extension ${\rm Ad}_{s}:G\to {\rm Aut}(\g_{\C})$ is given by
\[{\rm Ad}_{s}(\exp (A) \exp(X))=\exp( ({\rm ad}_{A})_{s})=(\exp({\rm ad}_{A}))_{s}
\]
and we have ${\rm Ad}_{s}(G)=\{(\exp({\rm ad}_{A}))_{s}\in {\rm Aut}(\g_{\C})\vert A\in V\}$.
Let $\bar G={\rm Ad}_{s}(G)\ltimes G$.
Then the Lie algebra of $\bar G$ is $\bar\g$.
For the nilradical $\bar N$ of $\bar G$, by the spritting $\bar \g= {\rm Im} \,{\rm ad}_{s}\ltimes \bar{\n}$  we have $\bar G={\rm Ad}_{s}(G)\ltimes \bar N$  such that we can regard ${\rm Ad}_{s}(G)\subset \rm Aut (\bar N)$ and  ${\rm Ad}_{s}(G)$ consists of semi-simple automorphisms of $\bar N$.
By the construction of $\bar\n$ we have  $\bar G=G\cdot \bar N$.

A simply connected nilpotent Lie group is considered as the real points of a unipotent $\R$-algebraic group (see \cite[p. 43]{OV}) by the exponential map.
We have the unipotent $\R$-algebraic group $\bf \bar{N}$ with ${\bf \bar{N}}(\R)=\bar{N}$.
We identify $\Aut_{a}({\bf \bar{N}})$ with $\Aut(\n_{\C})$ and $\Aut_{a}({\bf \bar{N}})$ has the $\R$-algebraic group structure with ${\rm Aut}_{a}({\bf \bar{N}})(\R)= {\rm Aut} (\bar N)$.
So we have the $\R$-algebraic group $ \Aut_{a} (\bf\bar{N})\ltimes \bar{N}$.
By ${\rm Ad}_{s}(G)\ltimes G={\rm Ad}_{s}(G)\ltimes \bar N$,
we have the injection $I:G\to {\rm Aut} (\bar N)\ltimes \bar N =\Aut_{a} (\bf \bar{N})\ltimes \bar{N}(\R)$.
Let $\bf G$ be the Zariski-closure of $I(G)$ in $ \Aut_{a} (\bf  \bar{N})\ltimes \bar{N}$.

\begin{proposition}\label{hulll}\label{u=n}
Let $\bf T$ be the Zariski-closure of ${\rm Ad}_{s}(G)$ in ${\rm Aut}\, \bar{\bf N}$.
Then we have ${\bf G}={\bf T}\ltimes \bf \bar{N}$ and ${\bf G}$ is the algebraic hull of $G$ with the unipotent hull ${\bf U}_{G}=\bf \bar{N}$.
Hence the Lie algebra of unipotent hull ${\bf U}_{G}$ of $G$ is 
 \[\bar\n_{\C} =\{X-{\rm ad}_{sX}\in  \bar\g_{\C} \vert X\in\g_{\C}\}.\]
\end{proposition}
\begin{proof}
The algebraic group ${\bf T}\ltimes \bar{\bf N}$ is the Zariski-closure of ${\rm Ad}_{s}(G)\ltimes \bar{N}$ in ${\rm Aut}(\bar{N})\ltimes\bar{N}$.
By   ${\rm Ad}_{s}(G)\cdot I(G)={\rm Ad}_{s}(G)\ltimes \bar{N}$, we have $ {\bf T}\cdot {\bf G}= {\bf T}\ltimes  \bar{\bf N}$.
Since $\bf T$ is a diagonalizable algebraic group, we have $\bar{\bf N}\subset {\bf G}$.
Otherwise since   $\bf G \subset {\bf T}\ltimes \bar{\bf N}$ is a  connected solvable algebraic group, we have ${\bf U}({\bf G})=\bar{\bf N}\cap {\bf G}=\bar{\bf N}$.
Since  we have ${\rm Ad}_{s}(G)\ltimes \bar{N}=G\cdot\bar N$,
 $\bf G$ is identified with the Zariski-closure of ${\rm Ad}_{s}(G)\ltimes \bar{N}$.
Hence we have ${\bf G}={\bf T}\ltimes \bf \bar{N}$.
By $\dim G=\dim \bar N$, we can easily check that ${\bf T}\ltimes \bf \bar{N}$ is the algebraic hull of $G$.
\end{proof}
By this proposition the Zariski-closure $\bf T$ of ${\rm Ad}_{s}(G)$ is a maximal torus of the algebraic hull of $G$.
By the uniqueness of the algebraic hull (see \cite[Lemma 4.41]{R}), we have:
\begin{lemma}\label{inde}
Let ${\bf H}_{G}$ be the algebraic hull of $G$ and $q:{\bf H}_{G}\to {\bf H}_{G}/{\bf U}_{G}$ the quotient map.
Then for any injection $\psi :G\to {\bf H}_{G}(\R)  $ as in Proposition \ref{ttt}, there exists an isomorphism  $\varphi:{\bf H}_{G}/{\bf U}_{G}\to {\bf T}$ such that the diagram
\[\xymatrix{
{\bf H}_{G}/{\bf U}_{G} \ar[r]^{\varphi}&{\bf T}  \\
	G\ar^{q\circ\psi}[u]	 \ar_{{\rm Ad}_{s}}[ru]
 }
\] 
commutes.
\end{lemma}

\begin{lemma}\label{HGHG}
Let $H_{G}={\bf H}_{G}(\R)$ be the real points of the algebraic hull of $G$.
Let $\bf T$ be the Zariski-closure of ${\rm Ad}_{s}(G)$ in ${\rm Aut}(\g_{\C})$ and $T={\bf T}(\R)$ its real points.
Then we have a semi-direct product
\[H_{G}=T\ltimes G .
\] 
\end{lemma}
\begin{proof}
By ${\rm Im}({\rm ad}_{s}) \ltimes\bar{\n}={\rm Im} ({\rm ad}_{s}) \ltimes\g$, we have  ${\rm Ad}_{s}(G)\ltimes {\bar N}= {\rm Ad}_{s}(G)\ltimes I(G)$.
Hence the lemma follows from Proposition \ref{u=n}.
\end{proof}

\begin{proposition}\label{abab}{\rm(\cite{K})}
Let $G$ be a simply connected solvable Lie group.
Then ${\bf U}_{G}$ is abelian if and only if  $G=\R^{n}\ltimes_{\phi} \R^{m}$ such that the action $\phi:\R^{n}\to {\rm  Aut} (\R^{m})$ is semi-simple.
\end{proposition}
\begin{proof}
Suppose  ${\bf U}_{G}$ is abelian.
Then by proposition \ref{hulll}, the Lie algebra $\bar \n$ is abelian.
By  $\n\subset \bar{\n}$, the nilradical $\n$ of $\g$ is abelian.
By $[\g,\g]\subset \n$, $\g$ is two-step solvable.
We consider the lower central series $\g^{i}$ as $\g^{0}=\g$ and $\g^{i}=[\g,\g^{i-1}]$ for $i\ge 1$.
We denote $\n^{\prime}=\cap^{\infty}_{i=0} \g^{i}$.
Then
by \cite[Lemma 4.1]{De}, we have $\g=\g/\n^{\prime}\ltimes \n^{\prime}$.
For this decomposition, the subspace $\{X-{\rm ad}_{sX}\vert X\in \g/\n^{\prime}\}\subset \bar\n$ is a Lie subalgebra of $\bar\n$.
Since $\g/\n^{\prime}$ is a nilpotent subalgebra of $\g$, this space is identified with  $\g/\n^{\prime}$.
Thus since $\bar \n$ is abelian,  $\g/\n^{\prime}$ is also abelian.
We show that the action of $\g/\n$ on $\n$ is semi-simple.
Suppose for some $ X\in \g/\n^{\prime}$  ${\rm ad}_{X}$ on $\n$ is not semi-simple .
Then ${\rm ad}_{X}-{\rm ad}_{sX}$ on $\n$ is not trivial.
Since we have $\bar{\n}=\{X-{\rm ad}_{sX}\vert X\in \g\}$, we have $[\bar{\n},\n]\not= \{0\}$.
This contradicts $\bar{\n}$ is abelian. 
Hence the action of $\g/\n$  on $\n$ is semi-simple.
Hence the first half of the proposition follows.
The converse follows from Lemma \ref{semmm}.
\end{proof}

\section{Left-invariant forms and the cohomology of solvmanifolds}
Let $G$ be a simply connected solvable Lie group, $\g$ the Lie algebra of $G$ and $\rho: G\to GL(V_{\rho})$ a representation on a $\C$-vector space $V_{\rho}$.
We consider the cochain complex $\bigwedge \g^{\ast}$ with the  derivation $d$ which is the dual to the Lie bracket of $\g$.
Then $\bigwedge \g^{\ast}_{\C}\otimes V_{\rho}$ is a cochain complex with the derivation 
$d_{\rho}=d+\rho_{\ast}$ where $\rho_{\ast}$ is the derivation of $\rho$ and consider $\rho_{\ast}\in \g^{\ast}_{\C}\otimes {\frak gl}(V_{\rho})$.
We can consider the cochain complex   $(\bigwedge \g^{\ast}_{\C}\otimes V_{\rho}, d_{\rho})$
 the twisted $G$-invariant differential forms on $G$.
Consider the cochain complex $A^{\ast}_{\C}(G)\otimes V_{\rho}$ with the derivation $d$ such that
\[d(\omega\otimes v)=(d\omega ) \otimes v \  \ \  \ \omega\in A^{\ast}_{\C}(G), \ \  \ v\in V_{\rho}.
\]
By the left action of $G$ (given by $(g\cdot f)(x)=f(g^{-1}x)$, \  $f\in C^{\infty}(G)$, $g\in G$) 
 and $\rho$, we have the action of $G$ on $A^{\ast}_{\C}(G)\otimes V_{\rho}$.
Denote $(A^{\ast}_{\C}(G)\otimes V_{\rho})^{G}$ the $G$-invariant elements of $A^{\ast}_{\C}(G)\otimes V_{\rho}$.
Then we have an isomorphism 
\[(A^{\ast}_{\C}(G)\otimes V_{\rho})^{G}\cong \bigwedge \g^{\ast}_{\C}\otimes V_{\rho}.
\]

Suppose $G$ has a lattice $\Gamma$.
Since $\pi_{1}(G/\Gamma)=\Gamma$, we have a flat vector bundle $E_{\rho_{\vert_{\Gamma}}}$ with flat connection $D_{\rho_{\vert_{\Gamma}}}$ on $G/\Gamma$ whose monodromy  is $\rho_{\vert_{\Gamma}}$.
Let $A^{\ast}(G/\Gamma, E_{\rho_{\vert_{\Gamma}}})$ be the cochain complex of $E_{\rho_{\vert_{\Gamma}}}$-valued differential forms with the  derivation $D_{\rho_{\vert_{\Gamma}}}$.
Consider the cochain complex $A^{\ast}_{\C}(G)\otimes V_{\rho}$ with derivation $d$ such that
\[d(\omega\otimes v)=(d\omega ) \otimes v \  \ \  \ \omega\in A^{\ast}_{\C}(G), \ \  \ v\in V_{\rho}.
\]
Then we have the $G$-action on  $A^{\ast}_{\C}(G)\otimes V_{\rho}$ and denote 
 $(A^{\ast}_{\C}(G)\otimes V_{\rho})^{\Gamma}$ the subcomplex of $\Gamma$-invariant elements of  $A^{\ast}_{\C}(G)\otimes V_{\rho}$.
We have the isomorphism  $(A^{\ast}_{\C}(G)\otimes V_{\rho})^{\Gamma}\cong A^{\ast}(G/\Gamma, E_{\rho_{\vert_{\Gamma}}})$.
Thus we have 
\[\bigwedge \g^{\ast}_{\C}\otimes V_{\rho}\cong (A^{\ast}_{\C}(G)\otimes V_{\rho})^{G} \subset(A^{\ast}_{\C}(G)\otimes V_{\rho})^{\Gamma}\cong A^{\ast}(G/\Gamma, E_{\rho_{\vert_{\Gamma}}})\ 
\]
and we have the inclusion $\bigwedge \g^{\ast}_{\C}\otimes V_{\rho}\to A^{\ast}(G/\Gamma, E_{\rho_{\vert_{\Gamma}}})$.

We call a representation $\rho$ $\Gamma$-admissible if for the representation $\rho\oplus {\rm Ad} :G\to GL_{n}(\C)\times {\rm Aut}(\g_{\C})$,  
$(\rho\oplus {\rm Ad})(G)$ and $(\rho\oplus {\rm Ad})(\Gamma)$ have the same Zariski-closure in $GL_{n}(\C)\times {\rm Aut}(\g_{\C})$.
\begin{theorem}\label{Moss}{\rm (\cite{Mosc},\cite[Theorem 7.26]{R})} 
If $\rho$ is $\Gamma$-admissible, then the inclusion 
\[\bigwedge \g^{\ast}_{\C}\otimes V_{\rho}\to A^{\ast}(G/\Gamma, E_{\rho_{\vert_{\Gamma}}})
\]
induces a cohomology isomorphism.
\end{theorem}

\begin{proposition}\label{adddm}
Let $G$ be a simply connected solvable Lie group with a lattice $\Gamma$.
We suppose ${\rm Ad}(G)$ and ${\rm Ad}(\Gamma)$ have the same Zariski-closure in ${\rm Aut}(\g_{\C})$.
We consider the diagonalizable representation ${\rm Ad}_{s}:G\to {\rm Aut}(G)$.
Let $\bf T$ be the Zariski-closure of ${\rm Ad}_{s}(G)$ and $\alpha$ be a character of $\bf T$.
Then $\alpha\circ {\rm Ad}_{s}$ is $\Gamma$-admissible.
\end{proposition}
\begin{proof}
Let $\bf G$ be the Zariski-closure of ${\rm Ad}(G)$ in ${\rm Aut}(\g_{\C})$.
We first show that $\bf T$ is a maximal torus of $\bf G$.
For the direct sum $\g=V\oplus \n$ as Construction \ref{con}, the map  $F:V\oplus \n\to G$ defined by $F(A+X)=\exp (A) \exp(X)$ for $A\in V$, $X\in \n$ is a diffeomorphism (see \cite[Lemaa 3.3]{Dek}).
For  $A\in V$, we consider   the Jordan decomposition ${\rm Ad}(\exp (A))=\exp (({\rm ad}_{A})_{s}) \exp (({\rm ad}_{A})_{n})$.
Then we have $\exp (({\rm ad}_{A})_{s}), \exp (({\rm ad}_{A})_{n})\in {\bf G}$.
For  $X\in \n$, we have $\exp ({\rm ad}_{X})\in {\bf U}({\bf G})$.
Hence we have ${\rm Ad}(G)\subset {\bf T}{\bf U}({\bf G})\subset {\bf G}$.
 Since ${\bf G}$ is Zariski-closure of ${\rm Ad}(G)$,  ${\bf G}={\bf T}{\bf U}({\bf G})$.
Thus  $\bf T$ is a maximal torus of $\bf G$.

We take a spritting ${\bf G}={\bf T}\ltimes {\bf U}({\bf G})$.
We consider the algebraic group
\[{\bf G}^{\prime}=\{(\alpha(t), (t, u))\in \C^{\ast}\times {\bf G}\vert (t,u)\in {\bf T}\ltimes {\bf U}({\bf G})\}.
\]
Then we have
\begin{multline*}
(\alpha\circ {\rm Ad}_{s}\oplus {\rm Ad})(G)\\
=\{(\alpha(\exp (({\rm ad}_{A})_{s})), \exp ({\rm ad}_{A})\exp({\rm ad}_{X}))\vert A+X\in V\oplus \n\}\\
\subset {\bf G}^{\prime}.
\end{multline*}
Since ${\bf G}$ is Zariski-closure of ${\rm Ad}(G)$,   $(\alpha\circ {\rm Ad}_{s}\oplus {\rm Ad})(G)$ is Zariski-dense in ${\bf G}^{\prime}$.
Since $  {\rm Ad}(G)$ and ${\rm Ad}(\Gamma)$ have the same Zariski-closure, $(\alpha\circ {\rm Ad}_{s}\oplus {\rm Ad})(G)$ and $(\alpha\circ {\rm Ad}_{s}\oplus {\rm Ad})(\Gamma)$ have the same Zariski-closure ${\bf G}^{\prime}$.
\end{proof}

\section{Hain's DGAs}
\subsection{Constructions}
Let $M$ be a $C^{\infty}$-manifold, $\bf S$ be a reductive algebraic group and $\rho :\pi_{1}(M,x)\to {\bf S}$ be a representation.
We assume the image of $\rho$ is Zariski-dense in $\bf S$.
Let $\{V_{\alpha}\}$ be the set of irreducible representations  of $\bf S$ and $(E_{\alpha}, D_{\alpha})$ be a  flat bundle with the monodromy $\alpha \circ \rho$ and $A^{\ast}(M,E_{\alpha})$  the space of $E_{\rho}$-valued $C^{\infty}$-differential forms.
Then we have an algebra isomorphism of $\bigoplus_{\alpha}V_{\alpha}\otimes V_{\alpha}^{\ast}$ and the coordinate ring $\C[{\bf S}]$ of $\bf S$ (see \cite[Section 3]{Hai}).
Denote 
\[A^{\ast}(M,{\mathcal O}_{\rho})=\bigoplus_{\alpha} A^{\ast}(M, E_{\alpha})\otimes V_{\alpha}^{\ast}\]
 and $D=\bigoplus_{\alpha} D_{\alpha}$.
Then by the wedge product, $(A(M, {\mathcal O}_{\rho})), D)$ is a cohomologically connected  DGA with coefficients in $\C$ .

Suppose $\bf S$ is a diagonal algebraic group.
Then  $\{V_{\alpha}\}$ is the set of one-dimensional representations for all algebraic characters $\alpha$ of $\bf T$ and $(E_{\alpha}, D_{\alpha})$ are rank one flat bundles with the monodromy $\alpha \circ \rho$.
In this case for characters $\alpha$ and $\beta$, we have
 the wedge product $A^{\ast}(M,E_{\alpha})\otimes A^{\ast}(M,E_{\beta})\to A^{\ast}(M,E_{\alpha \beta})$ and $D_{\alpha \beta}(\psi_{\alpha}\wedge\psi_{\beta})=D_{\alpha}\psi_{\alpha}\wedge \psi_{\beta}+(-1)^{p}\psi_{\alpha}\wedge D_{\beta}\psi_{\beta}$ for $\psi_{\alpha}\in A^{p}(M,E_{\alpha}),\,\psi_{\beta}\in A^{q}(M,E_{\beta})$ (see \cite{Nar} for  details in this case).

\subsection{Formality and the hard Lefschetz properties of compact K\"ahler manifolds}
In this subsection we will prove the following theorem by theories of Higgs bundles studied by Simpson.
\begin{theorem}\label{fol}
Let $M$ be a compact K\"ahler manifold with a K\"ahler form $\omega$ and $\rho:\pi_{1}(M)\to {\bf S}$ a representation to a reductive algebraic group $\bf S$ with the Zariski-dense imaga.
Then the following conditions hold:\\
{\rm (A)} (formality) The DGA  $A^{\ast}(M,{\mathcal O}_{\rho}))$ is formal.\\
{\rm (B)}(hard Lefschetz) For any $0\leq i\leq n$ the  linear operator
\[ [\omega]^{i}\wedge: H^{n-i}(A^{\ast}(M,{\mathcal O}_{\rho}))) \to H^{n+i}(A^{\ast}(M,{\mathcal O}_{\rho}))) \]
is an isomorphism where $\dim _{\R} M=2n$.

\end{theorem}

Let $M$ be a compact K\"ahler manifold and $E$ a holomorphic vector bundle on $M$ with the Dolbeault operator ${\bar \partial }$.
For a ${\rm End}( E)$-valued holomorphic form $\theta $, we denote $D^{\prime \prime}={\bar \partial }+\theta$.
We call $(E,D^{\prime \prime})$ a Higgs bundle if it satisfies the Leibniz rule: $ D^{\prime \prime}(ae)={\bar \partial }(a) e+(-1)^{p}D^{\prime \prime}(e)$ for $a\in A^{p}(M)$, $e\in A^{0}(E)$ and the integrability: $(D^{\prime \prime})^{2}=0$.
Let $h$ be a Hermitian metric on $E$.
For a Higgs bundle $(E, D^{\prime \prime}={\bar \partial }+\theta)$.
We define $D^{\prime}_{h}=\partial_{h}+{\bar \theta}_{h}$ as follows:
$\partial_{h}$ is the unique operator which satisfies 
\[h({\bar \partial }e,f)+h(e,\partial_{h}f)={\bar \partial }h(e,f)\]
and ${\bar \theta}_{h}$ is defined by $(\theta e, f)=(e, {\bar \theta}_{h}f)$.
Let $D_{h}=D^{\prime}_{h}+D^{\prime \prime}$.
Then $D_{h}$ is a connection.
We call a Higgs bundle $(E,D^{\prime \prime},h)$ with a metric harmonic if $D_{h}$ is flat i.e. $(D_{h})^{2}=0$.

For  two Higgs bundles $(E,D^{\prime \prime})$, $(F, D^{\prime \prime})$ with metric $h_{E}$, $h_{F}$, the tensor product $(E\otimes F, D^{\prime \prime}\otimes 1+1\otimes D^{\prime \prime})$ is an also Higgs bundle and $h_{E}\otimes h_{F}$ gives the connection $D_{h_{E}\otimes h_{F}}=D_{h_{E}}\otimes 1+ 1\otimes D_{h_{F}}$ on $E\otimes F$.
If $(E,D^{\prime \prime},h_{E})$ and $(F,D^{\prime \prime},h_{F})$
are harmonic, $(E\otimes F, D^{\prime \prime}\otimes 1+1\otimes D^{\prime \prime})$ is also a harmonic Higgs bundle with the flat connection $D_{h_{E}}\otimes 1+ 1\otimes D_{h_{F}}$.

\begin{theorem}{\rm (\cite[Theorem 1]{Sim}\label{sem})}
Let $(E,D)$ be a flat bundle on $M$ whose monodromy is semi-simple.
Then  $D$ is given by  a harmonic Higgs bundle  $(E, D^{\prime\prime},h)$ that is $D=D_{h}$.
\end{theorem}

\begin{theorem}{\rm(\cite[Lemma 2.2]{Sim})}\label{foomm}
Let $(E, D^{\prime\prime},h)$ be a harmonic Higgs bundle with the flat connection $D=D^{\prime}+D^{\prime\prime}$.
Then the inclusion \[
({\rm Ker}\, D^{\prime},D^{\prime\prime})\to (A^{\ast}(E),D)
\]
and the quotient  \[
({\rm Ker}\, D^{\prime},D^{\prime\prime})\to (H_{D^{\prime}}(A^{\ast}(E)),D^{\prime\prime})=(H_{D}^{\ast}(A^{\ast}(E)), 0)
\] 
induce the cohomology isomorphisms.

\end{theorem} 

\begin{theorem}{\rm(\cite[Lemma 2.6]{Sim})}\label{leff}
Let $(E, D^{\prime\prime},h)$ be a harmonic Higgs bundle with the flat connection $D=D^{\prime}+D^{\prime\prime}$. Then for any $0\leq i\leq n$ the  linear operator
\[ [\omega]^{n-i}\wedge: H_{D}^{i}(A^{\ast}(E_{\rho})) \to H_{D}^{2n-i}(A^{\ast}(E_{\rho})) \]
is an isomorphism.
\end{theorem}

\begin{proof}[Proof of Theorem \ref{fol}] \  \\
 By Theorem \ref{sem} and \ref{leff}, the condition (B) holds.
By Theorem \ref{sem}, for $(A^{\ast}(E_{\alpha}), D_{\alpha})$, we have $D_{\alpha}=D^{\prime}_{\alpha}+D^{\prime\prime}_{\alpha}$ such that $D^{\prime\prime}_{\alpha}$ is a harmonic Higgs bundle.
Denote $D^{\prime}=\bigoplus_{\alpha} D^{\prime}_{\alpha}$ and $D^{\prime\prime}=\bigoplus_{\alpha} D^{\prime\prime}_{\alpha}$.
Then by properties of Higgs bundle,
$({\rm Ker}\, D^{\prime},D^{\prime\prime})$ is  a DGA, and the maps \[({\rm Ker}\, D^{\prime},D^{\prime\prime})\to (A^{\ast}(M,{\mathcal O}_{\rho})),D)\] and  \[({\rm Ker}\, D^{\prime},D^{\prime\prime})\to (H_{D}^{\ast}(A^{\ast}(M,{\mathcal O}_{\rho})), 0)\] are DGA homomorphisms, thus quasi-isomorphisms by Theorem \ref{foomm}.
Hence the condition (A) holds.
\end{proof} 

\section{Minimal models of invariant forms on solvable Lie groups with local systems}

Let $G$ be a simply connected solvable Lie group and $\g$ the Lie algebra of $G$.
Consider the diagonal representation ${\rm Ad}_{s}$ as in Section 1 and the derivation ${\rm ad}_{s}$ of  ${\rm Ad}_{s}$.
For some basis $\{X_{1},\dots X_{n}\}$ of $\g_{\C}$,  ${\rm Ad}_{s}$ is represented by diagonal matrices.
Let $\bf T$ be the Zariski-closure of ${\rm Ad}_{s}(G)$ in ${\rm Aut}(\g_{\C})$.
Let $\{ V_{\alpha}\}$ be the set of one-dimensional representations for all characters $\alpha$ of $\bf T$.
We consider $V_{\alpha}$ the representation of $\g$ which is the derivation of $\alpha\circ {\rm Ad}_{s}$.
Then we have the cochain complex of Lie algebra $(\bigwedge \g^{\ast}_{\C}\otimes V_{\alpha}, d_{\alpha})$.
Denote $d=\bigoplus_{\alpha} d_{\alpha}$.
Then $( \bigoplus_{\alpha} \bigwedge \g^{\ast}_{\C}\otimes V_{\alpha}, d)$ is a cohomologically connected DGA with coefficients in $\C$ as the last section.
By ${\rm Ad}_{s}(G)\subset {\rm Aut}(\g_{\C})$ we have ${\bf T} \subset {\rm Aut}(\g_{\C})$ and hence we have the action of $\bf T$ on  $\bigoplus_{\alpha} \bigwedge \g^{\ast}_{\C}\otimes V_{\alpha}$.
Denote  $(\bigoplus_{\alpha} \bigwedge \g^{\ast}_{\C}\otimes V_{\alpha})^{\bf T}$ the sub-DGA of $\bigoplus_{\alpha} \bigwedge \g^{\ast}_{\C}\otimes V_{\alpha}$ which consists of the $\bf T$-invariant elements of $\bigoplus_{\alpha} \bigwedge \g^{\ast}_{\C}\otimes V_{\alpha}$.
\begin{lemma}\label{incso}
We have an isomorphism
\[H^{\ast}((\bigoplus_{\alpha} \bigwedge \g^{\ast}_{\C}\otimes V_{\alpha})^{\bf T})\cong H^{\ast}(\bigoplus_{\alpha} \bigwedge \g^{\ast}_{\C}\otimes V_{\alpha}).
\]
\end{lemma}
\begin{proof}

We show that the action  of ${\rm Ad}_{s}(G)\subset \bf T$ on the cohomology $H^{\ast}(  \bigwedge \g^{\ast}_{\C}\otimes V_{\alpha})$ is  trivial.
Consider the direct sum $\g=V\oplus \n$ as Construction \ref{con}.
Then we have ${\rm Ad}_{s}(G)={\rm Ad}_{s}(\exp (V))$ by Lemma \ref{exponnn}.
For $A\in V$, the action ${\rm Ad}_{s}(\exp (A))$ on the cochain complex $\bigwedge \g^{\ast}_{\C}\otimes V_{\alpha}$ is a semi-simple part of the action of $\exp (A)$ on $\bigwedge \g^{\ast}_{\C}\otimes V_{\alpha}$ via ${\rm Ad}\otimes \alpha\circ {\rm Ad}_{s}$.
Since the action of $G$ on the cohomology $H^{\ast}(  \bigwedge \g^{\ast}_{\C}\otimes V_{\alpha})$ via ${\rm Ad}\otimes \alpha\circ {\rm Ad}_{s}$ is  the extension of the Lie derivation on $H^{\ast}(  \bigwedge \g^{\ast}_{\C}\otimes V_{\alpha})$, this $G$-action on  $H^{\ast}(  \bigwedge \g^{\ast}_{\C}\otimes V_{\alpha})$ is trivial.
Hence for $A\in V$ the action of ${\rm Ad}_{s}(\exp (A))=(\exp ({\rm ad}_{A}))_{s}$ on the cohomology $H^{\ast}(  \bigwedge \g^{\ast}_{\C}\otimes V_{\alpha})$  is trivial.

Since $\bf T$ is the Zariski-closure of ${\rm Ad}_{s}(G)$ in ${\rm Aut}(\g_{\C})$ and the action of $\bf T$ on $\bigwedge \g^{\ast}_{\C}\otimes V_{\alpha}$ is algebraic, the action of $\bf T$ on  $H^{\ast}(  \bigwedge \g^{\ast}_{\C}\otimes V_{\alpha})$ is also trivial.
Since the action of $\bf T$ on $\bigwedge \g^{\ast}_{\C}\otimes V_{\alpha}$ is diagonalizable,
we have an isomorphism
\[H^{\ast}(  \bigwedge \g^{\ast}_{\C}\otimes V_{\alpha})\cong H^{\ast}(  \bigwedge \g^{\ast}_{\C}\otimes V_{\alpha})^{\bf T}\cong H^{\ast}(  (\bigwedge \g^{\ast}_{\C}\otimes V_{\alpha})^{\bf T}).
\]
Hence we have the lemma.
\end{proof}
Consider the unipotent hull ${\bf U}_{G}$ of $G$. 
Let ${\frak u}$ be the $\C$-Lie algebra of ${\bf U}_{G}$ and $\frak u^{\ast}$ the $\C$-dual space.
We consider the DGA $\bigwedge {\frak u}^{\ast}$ with coefficients in $\C$.
\begin{lemma}\label{uniso}
We have an isomorphism of DGA 
\[\bigwedge {\frak u}^{\ast} \cong (\bigoplus_{\alpha} \bigwedge \g^{\ast}_{\C}\otimes V_{\alpha})^{\bf T}.
\] 
\end{lemma}
\begin{proof}
Let $\{x_{1}, \dots , x_{n}\}$ be the dual of a basis $\{X_{1},\dots, X_{n}\}$ of $\g$ such that ${\rm Ad}_{s}$ is represented by diagonal matrices.
We define characters $\alpha_{i}$ as $t\cdot X_{i}=\alpha_{i}(t)X_{i}$ for $t\in \bf T$.
Then we have $t\cdot x_{i}=\alpha^{-1}_{i}(t)x_{i}$.
Hence the vector space $(\bigoplus_{\alpha} \bigwedge^{1} \g^{\ast}_{\C}\otimes V_{\alpha})^{\bf T}$ is spanned by $\{x_{1}\otimes v_{\alpha_{1}},\dots ,x_{n}\otimes v_{\alpha_{n}}\}$ where $V_{\alpha_{i}}\ni v_{\alpha_{i}}\not=0$.
For 
\[\omega=\sum_{i_{1},\dots i_{p}, \alpha}a_{i_{1},\dots i_{p}, \alpha}x_{i_{1}}\wedge\dots \wedge x_{i_{p}}v_{\alpha} \in (\bigoplus_{\alpha} \bigwedge^{p} \g^{\ast}_{\C}\otimes V_{\alpha})^{\bf T},\] 
since any $x_{i_{1}}\wedge\dots \wedge x_{i_{p}}v_{\alpha}$ is an eigenvector of the action of $\bf T$, if $a_{i_{1},\dots i_{p}, \alpha}\not=0$ then $x_{i_{1}}\wedge\dots \wedge x_{i_{p}}v_{\alpha}$ is also a $\bf T$-invariant element.
Since  we have
\[t\cdot  x_{i_{1}}\wedge\dots \wedge x_{i_{p}}=\alpha^{-1}_{i_{1}}(t)\dots \alpha^{-1}_{i_{p}}(t) x_{i_{1}}\wedge\dots \wedge x_{i_{p}}\]
 for $t\in \bf T$,
 we have 
 \[x_{i_{1}}\wedge\dots \wedge x_{i_{p}}\otimes v_{\alpha}= x_{i_{1}}v_{\alpha_{i_{1}}}\wedge\dots \wedge x_{i_{p}}v_{\alpha_{i_{p}}}.\]
Thus the DGA $(\bigoplus_{\alpha} \bigwedge \g^{\ast}_{\C}\otimes V_{\alpha})^{\bf T}$ is generated by $\{x_{1}\otimes v_{\alpha_{1}},\dots ,x_{n}\otimes v_{\alpha_{n}}\}$.
Consider the Maurer-Cartan equations
\[dx_{k}=-\sum _{ij}c^{k}_{ij}x_{i}\wedge x_{j}
\]
and denote ${\rm ad}_{sX_{i}}(X_{j})=a_{ij}X_{j}$.
Since ${\rm Ad}_{sg}(X_{k})=\alpha_{i}({\rm Ad}_{sg})X_{k}$ for $g\in G$,
we have $dv_{\alpha_{k}}=\sum^{n}_{i=1} {\rm ad}_{sX_{i}}(X_{k})x_{i}v_{\alpha_{k}}=\sum^{n}_{i=1} a_{ik}x_{i}v_{ \alpha_{k}}$.
Then we have 
\[d_{\alpha_{k}}(x_{k}\otimes v_{\alpha_{k}})=-\sum_{ij} (c^{k}_{ij}x_{i}\wedge x_{j}\otimes v_{\alpha_{k}}-a_{ik}x_{i}\wedge x_{k}\otimes v_{\alpha_{k}}).
\]
Hence the DGA $(\bigoplus_{\alpha} \bigwedge \g^{\ast}_{\C}\otimes V_{\alpha})^{\bf T}$ is isomorphic to a free DGA generated degree 1 elements $\{y_{1},\dots y_{n}\}$ such that 
\[d(y_{k})=-\sum_{ij}(c^{k}_{ij}y_{i}\wedge y_{j}-a_{ik}y_{i}\wedge y_{k}).
\]
Let $\frak h$ be the Lie algebra which is the dual of the free DGA $(\bigoplus_{\alpha} \bigwedge \g^{\ast}_{\C}\otimes V_{\alpha})^{\bf T}$ and $\{Y_{1},\dots, Y_{n}\}$ the dual basis of $\{y_{1},\dots y_{n}\}$.
It is sufficient to show ${\frak h}\cong {\frak u}$.
Then the bracket of $\frak h$ is given by
\[ [Y_{i}, Y_{j}]=\sum_{k}c^{k}_{ij}Y_{k}-a_{ij}Y_{j}+a_{ji}Y_{i}.
\]
Otherwise by Section 2.3, we have ${\frak u}\cong \{X-{\rm ad}_{sX}| X\in \g_{\C}\}\subset D(\g_{\C})\ltimes \g_{\C}$.
For the basis $\{ X_{1}-{\rm ad}_{sX_{1}},\dots, X_{n}-{\rm ad}_{sX_{n}}\}$ of $\frak u$, we have
\[[X_{i}-{\rm ad}_{sX_{i}}, X_{j}-{\rm ad}_{sX_{j}}]=\sum_{k}c^{k}_{ij}X_{k}-a_{ij}X_{j}+a_{ji}X_{i}.
\]
By $[\g,\g]\subset {\frak n}$, we have $[{\frak u},{\frak u}]\subset {\frak n}_{\C}$ where $\frak n$ is the nilradical of $\g$.
By this we have 
\[\sum_{k}c^{k}_{ij}X_{k}-a_{ij}X_{j}+a_{ji}X_{i}\in\n_{\C},
\]
and hence we have
\[ad_{s\sum_{k}c^{k}_{ij}X_{k}-a_{ij}X_{j}+a_{ji}X_{i}}=0.
\]
This gives 
\begin{multline*}
[X_{i}-{\rm ad}_{sX_{i}}, X_{j}-{\rm ad}_{sX_{j}}]\\
=\sum_{k}c^{k}_{ij}(X_{k}-{\rm ad}_{sX_{k}})-a_{ij}(X_{j}-{\rm ad}_{sX_{j}})+a_{ji}(X_{i}-{\rm ad}_{sX_{i}}).
\end{multline*}
This gives an isomorphism ${\frak h}\cong {\frak u}$.
Hence the lemma follows.
\end{proof}
 Since ${\rm Ad}_{s}(G)$ is Zariski-dense in $\bf T$, ${\rm Ad}_{s}(G)$-invariant elements are also $\bf T$-invariant.
In particular we have the following lemma.
\begin{lemma}\label{TIN}
Let $T={\bf T}(\R)$ be the real points of $\bf T$.
Then we have 
\[(\bigoplus_{\alpha} \bigwedge \g^{\ast}_{\C}\otimes V_{\alpha})^{ T}\cong (\bigoplus_{\alpha} \bigwedge \g^{\ast}_{\C}\otimes V_{\alpha})^{\bf T}\cong \bigwedge {\frak u}^{\ast}.\]
\end{lemma}
Later we use this lemma.

Denote $A^{\ast}(\g_{\C},{\rm ad}_{s})=\bigoplus_{\alpha} \bigwedge \g^{\ast}_{\C}\otimes V_{\alpha}$.
By  lemma \ref{incso} and \ref{uniso} we have:
\begin{theorem}\label{qua}
We have a quasi-isomorphism of DGAs
\[\bigwedge {\frak u}^{\ast} \to A^{\ast}(\g_{\C},{\rm ad}_{s}).
\] 
Thus
$\bigwedge {\frak u}^{\ast}$ is the minimal model  of $ A^{\ast}(\g_{\C},{\rm ad}_{s})$.
\end{theorem}

\section{Cohomology of  $A^{\ast}(G/\Gamma,{\mathcal O}_{{\rm Ad}_{s\vert_{\Gamma}}})$}

Consider the two DGA $A^{\ast}(\g_{\C},{\rm ad}_{s})$  and $A^{\ast}(G/\Gamma,{\mathcal O}_{{\rm Ad}_{s\vert_{\Gamma}}})$.
For any  character $\alpha$ of an algebraic group $\bf T$ which is the Zariski-closure of ${\rm Ad}_{s}(G) $ in ${\rm Aut}(\g_{\C})$.
We have the inclusion 
\[\bigwedge \g^{\ast}_{\C}\otimes V_{\alpha}\cong (A^{\ast}_{\C}(G)\otimes V_{\alpha})^{G} \subset(A^{\ast}_{\C}(G)\otimes V_{\alpha})^{\Gamma}\cong A^{\ast}(G/\Gamma, E_{\alpha \circ{\rm Ad}_{s\vert_{\Gamma}}})\ .
\]
Thus we have the morphism of DGAs 
\[\phi: A^{\ast}(\g_{\C},{\rm ad}_{s})\to A^{\ast}(G/\Gamma,{\mathcal O}_{{\rm Ad}_{s\vert_{\Gamma}}}).\]

\begin{proposition}\label{cohoin}
The morphism $\phi: A^{\ast}(\g_{\C},{\rm ad}_{s})\to A^{\ast}(G/\Gamma,{\mathcal O}_{{\rm Ad}_{s\vert_{\Gamma}}})$ is injective and
the induced map
\[\phi^{\ast}:H^{\ast}(A^{\ast}(\g_{\C},{\rm ad}_{s}))\to H^{\ast}(A^{\ast}(G/\Gamma,{\mathcal O}_{{\rm Ad}_{s\vert_{\Gamma}}}))
\]
is also injective.
\end{proposition}
\begin{proof}
Since $G$ has a lattice $\Gamma$, $G$ is unimodular(see \cite[Remark 1.9]{R}).
Choose a Haar measure $d\mu$ such that the volume of $G/\Gamma$ is $1$.
We define a map $\varphi_{\alpha}: (A^{\ast}_{\C}(G)\otimes V_{\alpha})^{\Gamma}\to \bigwedge \g^{\ast}_{\C}\otimes V_{\alpha}$ as 
\[\varphi _{\alpha}(\omega\otimes v_{\alpha})(X_{1},\dots X_{p})=\int_{G/\Gamma}\frac{\omega_{x}}{\alpha(x)}(X_{1},\dots X_{p}) d\mu \cdot v_{\alpha}
\]
for $\omega\otimes v_{\alpha}\in (A^{p}_{\C}(G)\otimes V_{\alpha})^{\Gamma}$, $X_{1},\dots , X_{p}\in \g_{C}$.
Then each $\varphi _{\alpha}$ is a morphism of cochain complexes and we have $\varphi_{\alpha}\circ \phi_{ \vert_{\bigwedge \g^{\ast}_{\C}\otimes V_{\alpha}}}={\rm id}_{\vert_{ \bigwedge \g^{\ast}_{\C}\otimes V_{\alpha}}}$ (see \cite[Remark 7.30]{R}).
Thus the restriction 
\[\phi^{\ast} :H^{\ast}(\bigwedge \g^{\ast}_{\C}\otimes V_{\alpha})\to H^{\ast}(A^{\ast}(G/\Gamma, E_{\alpha}))
\]
is injective.
By this it is sufficient to show that two distinct characters $\alpha, \beta$ with $ \alpha\circ {\rm Ad}_{s\vert_{\Gamma}}=\beta\circ {\rm Ad}_{s\vert_{\Gamma}}$ satisfy $\varphi_{\beta}\circ\phi_{\vert_{\bigwedge \g^{\ast}_{\C}\otimes V_{\alpha}}}=0$.
For $\omega\otimes v_{\alpha}\in \bigwedge \g^{\ast}_{\C}\otimes V_{\alpha}$, we have
\[\varphi_{\beta}\circ\phi_{\vert_{\bigwedge \g^{\ast}_{\C}\otimes V_{\alpha}}}(\omega\otimes v_{\alpha})=\int_{G/\Gamma}\frac{\alpha(x)}{\beta(x)}\omega_{x}(X_{1},\dots X_{p}) d\mu \cdot v_{\alpha}.\]
Since $\omega\in \bigwedge \g^{\ast}_{\C}$, $\omega_{x}(X_{1},\dots X_{p})$ is constant on $G/\Gamma$.
Let $\lambda=\frac{\beta}{\alpha}d(\frac{\alpha}{\beta})$.
Then $\lambda$ is a $G$-invariant form.
Choose $\eta \in  \bigwedge \g^{\ast}_{\C}$ such that $\lambda \wedge \eta=d\mu$.
Then we have
\[d\left(\frac{\alpha}{\beta}\eta\right)=\frac{\alpha}{\beta} \lambda \wedge \eta=\frac{\alpha}{\beta}d\mu.
\]
By $ \alpha\circ {\rm Ad}_{s\vert_{\Gamma}}=\beta\circ {\rm Ad}_{s\vert_{\Gamma}}$, $\frac{\alpha}{\beta}\eta$ is $\Gamma$-invariant and 
we can consider $\frac{\alpha}{\beta}\eta$ a differential form on $G/\Gamma$.
Hence by Stokes' theorem, we have
\begin{multline*}
\int_{G/\Gamma}\frac{\alpha(x)}{\beta(x)}\omega_{x}(X_{1},\dots X_{p}) d\mu=\omega (X_{1},\dots X_{p}) \int_{G/\Gamma}\frac{\alpha(x)}{\beta(x)} d\mu \\
=\omega (X_{1},\dots X_{p}) \int_{G/\Gamma} d\left(\frac{\alpha}{\beta}\eta\right)=0.
\end{multline*}
This prove the proposition.
\end{proof}

\begin{corollary}\label{ISOS}
Let $G$ be a simply connected solvable Lie group with a lattice $\Gamma$.
We suppose ${\rm Ad}(G)$ and ${\rm Ad}(\Gamma)$ have the same Zariski-closure in ${\rm Aut}(\g_{\C})$.
Then we have an isomorphism
\[H^{\ast}(A^{\ast}(G/\Gamma,{\mathcal O}_{{\rm Ad}_{s\vert_{\Gamma}}}))\cong H^{\ast}( A^{\ast}(\g_{\C},{\rm ad}_{s})).
\]
\end{corollary}
\begin{proof}
Let $\bf T$ be the Zariski-closure of ${\rm Ad}_{s}(G)$.
For any $1$-dimensional representation $V_{\alpha}$ of $\bf T$ given by a character $\alpha$ of $\bf T$, we consider a flat bundle $E_{\alpha}$ on $G/\Gamma$ given by the representation $\alpha\circ {\rm Ad}_{s}$ and the two cochain complex $A^{\ast}(G/\Gamma, E_{\alpha})$ and  $\bigwedge \g^{\ast}_{\C}\otimes V_{\alpha}$ as above.
Then since $\alpha\circ {\rm Ad}_{s}$ is $\Gamma$-admissible, by Theorem \ref{Moss} we have an isomorphism
\[H^{\ast}(\bigwedge \g^{\ast}\otimes V_{\alpha})\cong H^{\ast}(A^{\ast}(G/\Gamma, E_{\alpha})).\]
By the definitions of $A^{\ast}(G/\Gamma,{\mathcal O}_{{\rm Ad}_{s\vert_{\Gamma}}})$ and $ A^{\ast}(\g_{\C},{\rm ad}_{s})$ the corollary follows.
\end{proof}

\section{Extensions}\label{Extens}
In this Section we extend Corollary \ref{ISOS} to the case of general sovmanifolds.
To do this we consider infra-solvmanifolds which are generalizations of solvmanifolds.
\subsection{Infra-solvmanifold}
Let $G$ be a simply connected solvable Lie group.
We consider the affine transformation group ${\rm Aut}(G)\ltimes G$ and the projection $p:{\rm Aut}(G)\ltimes G\to{\rm Aut}(G)$.
Let $\Gamma\subset {\rm Aut}(G)\ltimes G$ be a discrete subgroup such that $p(\Gamma)$ is contained in a compact subgroup of ${\rm Aut}(G)$ and the quotient $G/\Gamma$ is compact.
We call $G/\Gamma$ an infra-solvmanifold.
\begin{theorem}\cite[Theorem 1.5]{B}\label{Bau}
For two infra-solvmanifolds $G_{1}/\Gamma_{1}$ and $G_{2}/\Gamma_{2}$, if $\Gamma_{1}$ is isomorphic to $\Gamma_{2}$, then $G_{1}/\Gamma_{1}$ is diffeomorphic to $G_{2}/\Gamma_{2}$.
\end{theorem}

\subsection{Extensions for infra-solvmanifolds}
Let $\Gamma$ be a torsion-free polycyclic group.
 and ${\bf H}_{\Gamma}$ be the algebraic hull.
Then there exists a finite  index normal subgroup $\Delta$ of $\Gamma$ and a simply connected solvable  subgroup $G$ of ${\bf H}_{\Gamma}$ such that $\Delta$ is a lattice of $G$, and $G$ and $\Delta$ have the same Zariski-closure in ${\bf H}_{\Gamma}$ (see \cite[Proposition 2.9]{B}). 
Since the Zariski-closure of $\Delta$ in ${\bf H}_{\Gamma}$ is finite index normal subgroup of ${\bf H}_{\Gamma}$,  this group is the algebraic hull ${\bf H}_{\Delta}$ of $\Delta$ by the properties in Proposition \ref{ttt}.
By ${\rm rank}\, \Gamma=\dim G$,  ${\bf H}_{\Delta}$ is also the algebraic hull ${\bf H}_{G}$ of $G$.
Hence we have the commutative diagram
\[\xymatrix{
	G \ar[r]& {\bf H}_{\Delta}(={\bf H}_{G}) \ar[r]&{\bf H}_{\Gamma}   \\
	\Delta\ar[u]\ar[ru]	 \ar[r]&\Gamma \ar[ru]
 }.
\] 
Since $\Delta$ is a finite index normal subgroup of $\Gamma$, by this diagram ${\bf H}_{\Delta}$ is a finite index normal subgroup of ${\bf H}_{\Gamma}$.
We suppose ${\bf H}_{\Gamma}/{\bf U}_{\Gamma}$ is diagonalizable.
Let $\bf T$ and ${\bf T}^{\prime}$ be  maximal daiagonalizable subgroups of ${\bf H}_{\Gamma}$ and ${\bf H}_{\Delta}$.
Then we have decompositions ${\bf H}_{\Gamma}={\bf T}\ltimes {\bf U}_{\Gamma}$, ${\bf H}_{\Delta}={\bf T}^{\prime}\ltimes {\bf U}_{\Gamma}$.
Since ${\bf T}/{\bf T}^{\prime}={\bf H}_{\Gamma}/{\bf H}_{\Delta}$ is a finite group,
we have a finite subgroup ${\bf T}^{\prime\prime}$ of $\bf T$ such that ${\bf T}={\bf T}^{\prime\prime} {\bf T}^{\prime}$(see \cite[Proposition 8.7]{Bor}).

\begin{lemma}\label{GAMM}
${\bf H}_{\Gamma}=\Gamma {\bf H}_{\Delta}$.
\end{lemma}
\begin{proof}
Consider the quotient $q:{\bf H}_{\Gamma}\to {\bf H}_{\Gamma}/{\bf H}_{\Delta}$.
Since $\Gamma$ is Zariski-dense in ${\bf H}_{\Gamma}$, $q(\Gamma)$ is Zariski-dense in ${\bf H}_{\Gamma}/{\bf H}_{\Delta}$.
Since ${\bf H}_{\Gamma}/{\bf H}_{\Delta}$ is a finite group, $q(\Gamma)={\bf H}_{\Gamma}/{\bf H}_{\Delta}$.
Thus we have 
\[\Gamma  {\bf H}_{\Delta}=\Gamma {\bf T}^{\prime}\ltimes {\bf U}_{\Gamma}={\bf T}^{\prime\prime} {\bf T}^{\prime}\ltimes {\bf U}_{\Gamma}={\bf H}_{\Gamma}.
\]
\end{proof}
Let $H_{\Gamma}={\bf H}_{\Gamma}(\R)$, $T^{\prime}={\bf T}^{\prime}(\R)$ and $ T^{\prime\prime} ={\bf T}^{\prime\prime} (\R)$.
Then by Lemma \ref{HGHG} and $ H_{G}=H_{\Delta}$,  we have $H_{\Gamma}=T^{\prime}T^{\prime\prime}\ltimes G$.
Hence we have $\Gamma\subset H_{\Gamma}\subset {\rm Aut}(G)\ltimes G$.
Since $\Delta$ is a lattice of $G$ and a finite index normal subgroup of $\Gamma$,
$\Gamma$ is a discrete subgroup of ${\rm Aut}(G)\ltimes G$ and $G/\Gamma$ is compact and hence an infra-sovmanifold.

\begin{theorem}\label{INFRAI}
Let $\Gamma$ be a torsion-free polycyclic group and  $\Gamma\to {\bf H}_{\Gamma}$ be the algebraic hull of $\Gamma$.
Suppose $ {\bf H}_{\Gamma}/ {\bf U}_{\Gamma}$ is diagonalizable.
Let ${\frak u}$ be the Lie algebra of ${\bf U}_{\Gamma}$.
Let $\rho$ be the composition 
\[\Gamma\to {\bf H}_{\Gamma}\to {\bf H}_{\Gamma}/ {\bf U}_{\Gamma}.
\]
Then we have a quasi-isomorphism
\[\bigwedge {\frak u}^{\ast}\to A^{\ast}(G/\Gamma, {\mathcal O}_{\rho}).\]
\end{theorem}
\begin{proof}
In this proof for a DGA $A$ with a group $G$-action, we denote $(A)^{G}$ the sub DGA which consists of $G$-invariant elements of $A$.
Consider decompositions ${\bf H}_{\Gamma}={\bf T}\ltimes {\bf U}_{\Gamma}$, ${\bf H}_{\Delta}={\bf T}^{\prime}\ltimes {\bf U}_{\Gamma}$ as above.
Let $\{V_{\alpha}\}$ be the set of $1$-dimensional representations of $\bf T$ for  all characters $ \alpha$ of ${\bf T}$.
Consider the DGA $\bigoplus_{\alpha} A^{\ast}(G)\otimes V_{\alpha}$ with the derivation $d$
given by 
\[d(\omega\otimes v_{\alpha})=(d\omega )\otimes v_{\alpha} \ \ \ \ \ \ \omega\in A^{\ast}(G), \ \ v_{\alpha}\in V_{\alpha}\]
and the products given by 
\[(\omega_{1}\otimes v_{\alpha})\wedge (\omega_{2}\otimes v_{\beta})=(\omega_{1}\wedge \omega_{2}) \otimes (v_{\alpha}\otimes v_{\beta}).
\]
Then by the definition, we have 
\[A^{\ast}(G/\Gamma, {\mathcal O}_{\rho})=(\bigoplus_{\alpha} A^{\ast}(G)\otimes V_{\alpha})^{\Gamma}.
\]
Let $\{V_{\alpha^{\prime}}\}$ and $\{ V_{\alpha^{\prime\prime}}\}$ be the sets of $1$-dimensional representations of ${\bf T}^{\prime}$ and ${\bf T}^{\prime\prime}$ for all characters $\alpha^{\prime}$ of ${\bf T}^{\prime}$ and $\alpha^{\prime\prime}$ of  ${\bf T}^{\prime\prime}$.
By ${\bf T}={\bf T}^{\prime}{\bf T}^{\prime\prime}$, we have $\{V_{\alpha}\}=\{V_{\alpha^{\prime}}\otimes V_{\alpha^{\prime\prime}}\}$.
Then we have
\[H^{\ast}(A^{\ast}(G/\Gamma, {\mathcal O}_{\rho}))=H^{\ast}\left(\left(\bigoplus_{\alpha^{\prime},\alpha^{\prime\prime}} A^{\ast}(G)\otimes (V_{\alpha^{\prime}}\otimes V_{\alpha^{\prime\prime}})\right)^{\Gamma}\right).
\]
Since $\Delta$ is a finite index normal subgroup of $\Gamma$, we have 
\begin{multline*}
H^{\ast}\left(\left(\bigoplus_{\alpha^{\prime},\alpha^{\prime\prime}} A^{\ast}(G)\otimes (V_{\alpha^{\prime}}\otimes V_{\alpha^{\prime\prime}})\right)^{\Gamma}\right)\\
\cong H^{\ast}\left( \left(\bigoplus_{\alpha^{\prime},\alpha^{\prime\prime}} A^{\ast}(G)\otimes (V_{\alpha^{\prime}}\otimes V_{\alpha^{\prime\prime}})\right)^{\Delta}\right)^{\Gamma/\Delta}.
\end{multline*}
Since $\Delta \subset {\bf T}^{\prime}\ltimes {\bf U}_{\Gamma}$, for a character $\alpha^{\prime\prime}$ of ${\bf T}^{\prime\prime}$ $\Delta$ acts trivially on $V_{\alpha^{\prime\prime}}$.
Hence we have
\begin{multline*}H^{\ast}\left( \left(\bigoplus_{\alpha^{\prime},\alpha^{\prime\prime}} A^{\ast}(G)\otimes (V_{\alpha^{\prime}}\otimes V_{\alpha^{\prime\prime}})\right)^{\Delta}\right)^{\Gamma/\Delta}\\
=H^{\ast}\left(\bigoplus_{\alpha^{\prime\prime}}\left(\bigoplus_{\alpha^{\prime}} A^{\ast}(G)\otimes V_{\alpha^{\prime}}\right)^{\Delta}\otimes V_{\alpha^{\prime\prime}}\right)^{\Gamma/\Delta}.
\end{multline*}
Since $\Delta$ is a lattice of $G$ and we assume that ${\rm Ad}(G)$ and ${\rm Ad}(\Delta)$ have the same Zariski-closure, by  Corollary \ref{ISOS} we have
\[H^{\ast}\left( \left(\bigoplus_{\alpha^{\prime}} A^{\ast}(G)\otimes V_{\alpha^{\prime}}\right)^{\Delta}\right)
\cong H^{\ast}\left(\left(\bigoplus_{\alpha^{\prime}} A^{\ast}(G)\otimes V_{\alpha^{\prime}}\right)^{G}\right).
\]
By Lemma \ref{incso} and \ref{TIN}, we have
\begin{multline*}H^{\ast}\left(\left(\bigoplus_{\alpha^{\prime}} A^{\ast}(G)\otimes V_{\alpha^{\prime}}\right)^{G}\right)
\cong 
H^{\ast}\left(\left(\left(\bigoplus_{\alpha^{\prime}} A^{\ast}(G)\otimes V_{\alpha^{\prime}}\right)^{G}\right)^{T^{\prime}}\right).
\end{multline*}
Hence we have 
\begin{multline*}H^{\ast}\left(\bigoplus_{\alpha^{\prime\prime}}\left(\bigoplus_{\alpha^{\prime}} A^{\ast}(G)\otimes V_{\alpha^{\prime}}\right)^{\Delta}\otimes V_{\alpha^{\prime\prime}}\right)^{\Gamma/\Delta}\\
\cong 
H^{\ast}\left(\bigoplus_{\alpha^{\prime\prime}}\left(\left(\bigoplus_{\alpha^{\prime}} A^{\ast}(G)\otimes V_{\alpha^{\prime}}\right)^{G}\right)^{T^{\prime}}\otimes V_{\alpha^{\prime\prime}}\right)^{\Gamma/\Delta}.
\end{multline*}
Since $H_{\Delta}=T^{\prime}\ltimes G$, we have
\begin{multline*}H^{\ast}\left(\bigoplus_{\alpha^{\prime\prime}}\left(\left(\bigoplus_{\alpha^{\prime}} A^{\ast}(G)\otimes V_{\alpha^{\prime}}\right)^{G}\right)^{T^{\prime}}\otimes V_{\alpha^{\prime\prime}}\right)^{\Gamma/\Delta}\\
=H^{\ast}\left(\bigoplus_{\alpha^{\prime\prime}}\left(\bigoplus_{\alpha^{\prime}} A^{\ast}(G)\otimes V_{\alpha^{\prime}}\right)^{H_{\Delta}}\otimes V_{\alpha^{\prime\prime}}\right)^{\Gamma/\Delta}
\\
\cong H^{\ast}\left(\left(\bigoplus_{\alpha^{\prime\prime}}\bigoplus_{\alpha^{\prime}} A^{\ast}(G)\otimes V_{\alpha^{\prime}}\otimes V_{\alpha^{\prime\prime}}\right)^{\Gamma H_{\Delta}}\right).
\end{multline*}
By Lemma \ref{GAMM}, we have 
\begin{multline*}
H^{\ast}\left(\left(\bigoplus_{\alpha^{\prime\prime}}\bigoplus_{\alpha^{\prime}} A^{\ast}(G)\otimes V_{\alpha^{\prime}}\otimes V_{\alpha^{\prime\prime}}\right)^{\Gamma H_{\Delta}}\right)\\
=
H^{\ast}\left(\left(\bigoplus_{\alpha^{\prime\prime}}\bigoplus_{\alpha^{\prime}} A^{\ast}(G)\otimes V_{\alpha^{\prime}}\otimes V_{\alpha^{\prime\prime}}\right)^{H_{\Gamma}}\right).
\end{multline*}
Since $H_{\Gamma}=T^{\prime}T^{\prime\prime}\ltimes G$, as above we have
\begin{multline*}H^{\ast}\left(\left(\bigoplus_{\alpha^{\prime\prime}}\bigoplus_{\alpha^{\prime}} A^{\ast}(G)\otimes V_{\alpha^{\prime}}\otimes V_{\alpha^{\prime\prime}}\right)^{H_{\Gamma}}\right)\\
=H^{\ast}\left(\left(\bigoplus_{\alpha^{\prime\prime}}\left(\left(\bigoplus_{\alpha^{\prime}} A^{\ast}(G)\otimes V_{\alpha^{\prime}}\right)^{G}\right)^{T^{\prime}}\otimes V_{\alpha^{\prime\prime}}\right)^{T^{\prime \prime}}\right).
\end{multline*}
Thus it is sufficient to show that the DGA 
\[\left(\bigoplus_{\alpha^{\prime\prime}}\left(\left(\bigoplus_{\alpha^{\prime}} A^{\ast}(G)\otimes V_{\alpha^{\prime}}\right)^{G}\right)^{T^{\prime}}\otimes V_{\alpha^{\prime\prime}}\right)^{T^{\prime \prime}}
\]
is isomorphic to $\bigwedge {\frak u}^{\ast}$.
By Lemma \ref{TIN} we have 
\[\left(\bigoplus_{\alpha^{\prime\prime}}\left(\left(\bigoplus_{\alpha^{\prime}} A^{\ast}(G)\otimes V_{\alpha^{\prime}}\right)^{G}\right)^{T^{\prime}}\otimes V_{\alpha^{\prime\prime}}\right)^{T^{\prime \prime}}\cong (\bigoplus_{\alpha^{\prime\prime}}\bigwedge {\frak u}^{\ast}\otimes V_{\alpha^{\prime\prime}})^{T^{\prime \prime}}.
\]
Now let $\bigwedge {\frak u}^{\ast}=\bigoplus_{\beta}A_{\beta^{\prime\prime}}$ be the weight decomposition of $T^{\prime\prime}$ for characters $\beta^{\prime\prime}$ of ${\bf T}^{\prime\prime}$.
Then we have 
\[ (\bigoplus_{\alpha^{\prime\prime}}\bigwedge {\frak u}^{\ast}\otimes V_{\alpha^{\prime\prime}})^{T^{\prime \prime}}
=(\bigoplus_{\beta^{\prime\prime}}\bigoplus_{\alpha^{\prime\prime}}A_{\beta^{\prime\prime}}\otimes V_{\alpha^{\prime\prime}} )^{T^{\prime\prime}}=
\bigoplus_{\alpha^{\prime\prime}}A_{(\alpha^{\prime\prime})^{-1}}\otimes V_{\alpha^{\prime\prime}}.
\]
It is easily seen  that 
\[\bigoplus_{\alpha^{\prime\prime}}A_{(\alpha^{\prime\prime})^{-1}}\otimes V_{\alpha^{\prime\prime}}\cong \bigwedge {\frak u}^{\ast}.
\]
Hence the theorem follows.
\end{proof}

Obviously a solvmanifold $G/\Gamma$ is a infra-solvmanifold with polycyclic fundamental group $\Gamma$.
Since $\bf T$ is the Zariski-closure of ${\rm Ad}_{s}(\Gamma)$ and diagonalizable, we have:
\begin{corollary}\label{GENN}
Let $G$ be a simply connected solvable Lie group with a lattice $\Gamma$ and ${\bf U}_{G}$ be the unipotent hull of $G$.
Let $\frak u$ be the Lie algebra of ${\bf U}_{G}$.
Then we have a quasi-isomorphism
\[\bigwedge {\frak u}^{\ast}\to A^{\ast}(G/\Gamma,{\mathcal O}_{{\rm Ad}_{s\vert_{\Gamma}}}).
\]
Thus $\bigwedge {\frak u}^{\ast}$ is the minimal model of $A^{\ast}(G/\Gamma,{\mathcal O}_{{\rm Ad}_{s\vert_{\Gamma}}})$.
\end{corollary} 

Consider the injection $\phi: A^{\ast}(\g_{\C},{\rm ad}_{s})\to A^{\ast}(G/\Gamma,{\mathcal O}_{{\rm Ad}_{s\vert_{\Gamma}}})$.
By Theorem \ref{qua}, Proposition \ref{cohoin} and above Corollary, $\phi: A^{\ast}(\g_{\C},{\rm ad}_{s})\to A^{\ast}(G/\Gamma,{\mathcal O}_{{\rm Ad}_{s\vert_{\Gamma}}})$ is a quasi-isomorphism.
Hence we have:
\begin{corollary}\label{ISOINV}
Let $G$ be a simply connected solvable Lie group with a lattice $\Gamma$.
Then we have an isomorphism
\[H^{\ast}(A^{\ast}(G/\Gamma,{\mathcal O}_{{\rm Ad}_{s\vert_{\Gamma}}}))\cong H^{\ast}( A^{\ast}(\g_{\C},{\rm ad}_{s})).
\]
\end{corollary}

We can apply this corollary to computations of the  untwisted de Rham cohomology of solvmanifolds by invariant forms.
We have an extension of Mostow's theorem(=Theorem \ref{Moss} ) for the untwisted cohomology.
\begin{corollary}\label{untt}
Let $G$ be a simply connected solvable Lie group with a lattice $\Gamma$.
Let $\bf T$ be the Zariski-closure of ${\rm Ad}_{s}(G)$ in ${\rm Aut}(\g_{\C})$.
Denote $A_{\Gamma}$ a set of characters of $\bf T$ such that for $\alpha\in A_{\Gamma}$ the restriction of $\alpha\circ {\rm Ad}_{s}$ on $\Gamma$ is trivial.
Consider the sub-DGA $\bigoplus_{\alpha\in A_{\Gamma}} \bigwedge \g^{\ast}_{\C}\otimes V_{\alpha}$ of $A^{\ast}(\g_{\C},{\rm ad}_{s})$.
Then we have a quasi-isomorphisms
\[\left(\bigoplus_{\alpha\in A_{\Gamma}} \bigwedge \g^{\ast}_{\C}\otimes V_{\alpha}\right)^{\bf T}\to\bigoplus_{\alpha\in A_{\Gamma}} \bigwedge \g^{\ast}_{\C}\otimes V_{\alpha}\to A^{\ast}_{\C}(G/\Gamma).
\]
Moreover the DGA $\left(\bigoplus_{\alpha\in A_{\Gamma}} \bigwedge \g^{\ast}_{\C}\otimes V_{\alpha}\right)^{\bf T}$ is a sub-DGA of $\bigwedge {\frak u}^{\ast}$.
\end{corollary}
\begin{proof}
Since we can consider $A^{\ast}_{\C}(G/\Gamma)=A^{\ast}(G/\Gamma, E_{\bf 1})$  for the trivial character $\bf 1$, $A^{\ast}_{\C}(G/\Gamma)$ is a sub-DGA of $A^{\ast}(G/\Gamma,{\mathcal O}_{{\rm Ad}_{s\vert_{\Gamma}}})$.
Then we have 
\[\phi^{-1}(A^{\ast}_{\C}(G/\Gamma))=\bigoplus_{\alpha\in A_{\Gamma}} \bigwedge \g^{\ast}_{\C}\otimes V_{\alpha}.\]
Since we define $A^{\ast}(G/\Gamma,{\mathcal O}_{{\rm Ad}_{s\vert_{\Gamma}}})=\bigoplus A^{\ast}(G/\Gamma, E_{\alpha\circ {\rm Ad}_{s\vert_{\Gamma} }})$ as a direct sum of cochain complexes and $\phi: A^{\ast}(\g_{\C},{\rm ad}_{s})\to A^{\ast}(G/\Gamma,{\mathcal O}_{{\rm Ad}_{s\vert_{\Gamma}}})$ is a quasi-isomorphism by Corollary \ref{ISOINV}, the restriction $\phi: \phi^{-1}(A^{\ast}_{\C}(G/\Gamma))\to A^{\ast}_{\C}(G/\Gamma)$ is also a quasi-isomorphism.
By Lemma \ref{incso}, the inclusion
\[\left(\bigoplus_{\alpha\in A_{\Gamma}} \bigwedge \g^{\ast}_{\C}\otimes V_{\alpha}\right)^{\bf T}\to\bigoplus_{\alpha\in A_{\Gamma}} \bigwedge \g^{\ast}_{\C}\otimes V_{\alpha}\]
is a quasi-isomorphism.
By Lemma \ref{uniso}, $\left(\bigoplus_{\alpha\in A_{\Gamma}} \bigwedge \g^{\ast}_{\C}\otimes V_{\alpha}\right)^{\bf T}$ is a sub-DGA of $\bigwedge {\frak u}^{\ast}$.
Hence the corollary follows.
\end{proof}

\section{Formality and hard Lefschetz properties}

In \cite{H}, Hasegawa proved the following theorem.
\begin{theorem}{\rm (\cite{H})}
Consider a DGA $\bigwedge {\frak n}^{\ast}$ which is the dual of a nilpotent Lie algebra $\frak n$.
Then $\bigwedge {\frak n}^{\ast}$ is formal if and only if $\frak n$ is abelian.

\end{theorem}

By Hasegawa's theorem, Theorem \ref{qua}, Proposition \ref{abab} and  Corollary \ref{GENN}, we have the following theorem.
\begin{theorem}\label{SOLFORM}
Let $G$ be a simply connected solvable Lie group.
Then the following conditions are equivalent:\\
$(A)$ The DGA $A^{\ast}(\g_{\C},{\rm ad}_{s})$ is formal \\
$(B)$  ${\bf U}_{G}$ is abelian.\\
$(C)$ $G=\R^{n}\ltimes_{\phi} \R^{m}$ such that the action $\phi:\R^{n}\to {\rm  Aut} (\R^{m})$ is semi-simple.\\
Moreover suppose $G$ has a lattice $\Gamma$.
Then the above three conditions are equivalent to the following condition:\\
$(D)$ $A^{\ast}(G/\Gamma,{\mathcal O}_{{\rm Ad}_{s\vert_{\Gamma}}})$ is formal.
\end{theorem}

In \cite{BG}, Benson and Gordon proved:
\begin{theorem}{\rm (\cite{BG}, see also \cite[Section 4.6.4]{Fe})}
Consider a DGA $\bigwedge {\frak n}^{\ast}$ which is the cochain complex of the dual of a nilpotent Lie algebra $\frak n$.
Suppose we have $[\omega]\in H^{2}(\bigwedge {\frak n}^{\ast})$ such that $[\omega]^{n}\not=0$ where $2n=\dim \frak n$.
Then for any $0\leq i\leq n$ the linear operator
\[[\omega]^{n-i}\wedge: H^{i}(\bigwedge {\frak n}^{\ast})\to H^{2n-i}(\bigwedge {\frak n}^{\ast})
\]
is an isomorphism if and only if $\frak n$ is abelian. 
\end{theorem}

By this theorem, we have:
\begin{theorem}\label{SOLHARD}
Let $G$ be a simply connected solvable Lie group.
Suppose $\dim G=2n$ and $G$ has an $G$-invariant symplectic form $\omega$.
Then the following conditions equivalent:\\
$(A)$\[[\omega]^{n-i}\wedge: H^{i}(A^{\ast}(\g_{\C},{\rm ad}_{s}))\to H^{2n-i}(A^{\ast}(\g_{\C},{\rm ad}_{s}))
\]
is an isomorphic for any $i\le n$.\\
$(B)$  ${\bf U}_{G}$ is abelian.\\
$(C)$  $G=\R^{n}\ltimes_{\phi} \R^{m}$ such that the action $\phi:\R^{n}\to {\rm  Aut} (\R^{m})$ is semi-simple.

Suppose $G$ has a lattice $\Gamma$ and $G/\Gamma$ has a symplectic form(not necessarily $G$-invariant)  $\omega$.
Then the  conditions $(B)$ and $(C)$ are equivalent to the following condition:\\
$(D)$\[[\omega]^{n-i}\wedge: H^{i}(A^{\ast}(G/\Gamma,{\mathcal O}_{{\rm Ad}_{s\vert_{\Gamma}}}))\to H^{2n-i}(A^{\ast}(G/\Gamma,{\mathcal O}_{{\rm Ad}_{s\vert_{\Gamma}}}))
\] 
is an isomorphism  for any $i\le n$ \\
\end{theorem}

For infra-solvmanifolds, by Theorem \ref{INFRAI} and  Proposition \ref{abab}   we have:
\begin{theorem}\label{INFORM}
Let $M$ be a infra-solvmanifold with the torsion-free polycyclic fundamental group $\Gamma$ 
and $\Gamma\to {\bf H}_{\Gamma}$ be the algebraic hull of $\Gamma$.
Suppose $ {\bf H}_{\Gamma}/ {\bf U}_{\Gamma}$ is diagonalizable.
Let $\rho$ be the composition 
\[\Gamma\to {\bf H}_{\Gamma}\to  {\bf H}_{\Gamma}/ {\bf U}_{\Gamma}.
\]
Then following conditions are equivalent:\\
$(A)$  $A^{\ast}(M, {\mathcal O}_{\rho})$ is formal.\\
$(B)$  ${\bf U}_{\Gamma}$ is abelian.\\
$(C)$  $M$ is finitely covered by a solvmanifold $G/\Gamma$ such that
$G=\R^{n}\ltimes_{\phi} \R^{m}$ with a semi-simple action $\phi:\R^{n}\to {\rm  Aut} (\R^{m})$  and $\Gamma$ is a lattice of $G$.\\
If $\dim M=2n $ and $M$ has a symplectic form $\omega$, the conditions $(A)$, $(B)$ and $(C)$ are equivalent to the following condition:\\
$(D)$ \[[\omega]^{n-i}\wedge: H^{i}(A^{\ast}(M,{\mathcal O}_{\rho}))\to H^{2n-i}(A^{\ast}(M,{\mathcal O}_{\rho}))
\] 
is an isomorphism  for any $i\le n$.
\end{theorem}

\section{Examples and  remarks} 
Let $G$ be a simply connected solvable Lie group with a lattice $\Gamma$.
Suppose ${\bf U}_{G}$ is abelian.
In \cite{K} the author showed that $G/\Gamma$ is formal and if $G/\Gamma$ has a symplectic form, then $G/\Gamma$ is hard Lefschetz.
But the converses of these results are not true.
See the following examples.

\begin{example}(\cite{Saw}) \\
We consider a $8$-dimensional solvable Lie group $G=G_{1}\times \R$ such that : $G_{1}$ is the matrix group as
\begin{multline*}
\left \{\left( \begin{array}{ccccccc}
e^{a_{1}t}&0&0&0&0&e^{-a_{3}t}x_{2}&z_{1}\\
0&e^{a_{2}t}&0&e^{-a_{1}t}x_{3}&0&0&z_{2}\\
0&0&e^{a_{3}t}&0&e^{-a_{2}t}x_{1}&0&z_{3}\\
0&0&0&e^{-a_{1}t}&0&0&x_{1}\\
0&0&0&0&e^{-a_{2}t}&0&x_{2}\\
0&0&0&0&0&e^{-a_{3}t}&x_{3}\\
0&0&0&0&0&0&1
\end{array}
\right) :t, x_{i}, y_{i}\in\R\right\},
\end{multline*}
where $a_{1}, a_{2}, a_{3}$ are distinct real numbers such that $a_{1}+a_{2}+a_{3}=0$.

Let $\g$ be the Lie algebra of $G$ and $\g^{\ast}$ the dual of $\g$.
The cochain complex $(\bigwedge \g^{\ast}, d)$ is generated by a basis $\{\alpha ,\beta, \zeta _{i}, \eta _{i}\}$ of $\g^{\ast}$ such that:
\[d\alpha =0,\,\, d\beta =0,\]
\[d\zeta _{i}= a_{i}\alpha \wedge \zeta _{i},\]
\[d\eta _{1}=-a_{1}\alpha \wedge \eta _{1}- \zeta _{2}\wedge\zeta _{3},\]
\[d\eta _{2}=-a_{2}\alpha \wedge \eta _{2}- \zeta _{3}\wedge\zeta _{1},\]
\[d\eta _{3}=-a_{3}\alpha \wedge \eta _{3}- \zeta _{1}\wedge\zeta _{2}.\]
In \cite{Saw} Sawai showed that for some $a_{1}, a_{2}, a_{3}$, $G$ has a lattice $\Gamma$ and $G/\Gamma$ satisfies formality and has a $G$-invariant symplectic form 
\[\omega=\alpha\wedge \beta +p(\zeta _{1}\wedge \eta _{1}-\zeta _{2}\wedge \eta _{2})+q(-\zeta _{2}\wedge \eta _{2}+\zeta _{3}\wedge \eta _{3})\]
 satisfying the hard Lefschetz property where $pq\not=0$ and $p+q\not=0$.
We have
\[{\rm Ad}_{s}(G)=
\left \{\left( \begin{array}{cccccccc}
e^{a_{1}t}&0&0&0&0&0&0&0\\
0&e^{a_{2}t}&0&0&0&0&0&0\\
0&0&e^{a_{3}t}&0&0&0&0&0\\
0&0&0&e^{-a_{1}t}&0&0&0&0\\
0&0&0&0&e^{-a_{2}t}&0&0&0\\
0&0&0&0&0&e^{-a_{3}t}&0&0\\
0&0&0&0&0&0&1&0\\
0&0&0&0&0&0&0&1
\end{array}
\right) :t\in\R\right\}.
\]
Let $\bf T$ is the Zariski closure of ${\rm Ad}_{s}(G)$.
Then for some characters $\alpha_{1}, \alpha_{2},\alpha_{3}$ of  $\bf T$ ,
the cochain complexes $(\bigwedge \g^{\ast} \otimes V_{\alpha_{i}}, d_{\alpha_{i}})$  are given by:
\[d_{\alpha_{i}}(v_{\alpha_{i}})=-a_{i}\alpha \otimes v_{\alpha_{i}}
\] 
for $v_{\alpha_{i}}\in V_{\alpha_{i}}$.

We have
\[d_{\alpha_{2}}(\zeta_{2}\otimes v_{\alpha_{2}})=a_{2}\alpha \wedge\zeta _{2}\otimes v_{\alpha_{2}}+\zeta _{2}\wedge a_{2}\alpha \otimes v_{\alpha_{2}}=0,\]
\[d_{\alpha_{3}}(\zeta_{3} \otimes v_{\alpha_{3}})=a_{3}\alpha \wedge\zeta _{3}\otimes v_{\alpha_{3}}+\zeta _{3}\wedge a_{3}\alpha\otimes v_{\alpha_{3}}=0,\]
\begin{multline*}
d_{\alpha_{2}\alpha_{3}}(\eta_{1}\otimes v_{\alpha_{2}\alpha_{3}})\\=
-(a_{1}+a_{2}+a_{3})\alpha \wedge\eta_{1}\otimes v_{\alpha_{2}\alpha_{3}}-\zeta _{2}\wedge\zeta _{3}\otimes v_{\alpha_{2}\alpha_{3}}
\\ =-\zeta _{2}\wedge\zeta _{3}\otimes v_{\alpha_{2}\alpha_{3}}.
\end{multline*}
Hence in $H^{2}(\bigwedge \g_{\C}^{\ast} \otimes V_{\alpha_{2}\alpha_{3}})$,
\[[\zeta_{2} \otimes v_{\alpha_{2}}]\cdot [\zeta_{3}\otimes v_{\alpha_{3}}]=0\]
and we have the Massey triple product
\[\langle[\zeta_{2}\otimes v_{\alpha_{2}} ], [\zeta_{3}\otimes v_{\alpha_{3}}],  [\zeta _{3}\otimes v_{\alpha_{3}}]\rangle=[\eta _{1}\wedge \zeta _{3}\otimes v_{\alpha_{2}\alpha^{2}_{3}}] \]
in the quotient of 
\[
H^{2}(\bigwedge \g_{\C}^{\ast} \otimes V_{\alpha_{2}\alpha^{2}_{3}})\]
by \[
([\zeta_{2}\otimes  v_{\alpha_{2}}]\cdot H^{1}(\bigwedge \g_{\C}^{\ast} \otimes V_{\alpha^{2}_{3}})+[\zeta _{3}\otimes v_{\alpha_{3}}]\cdot H^{1}(\bigwedge \g_{\C}^{\ast} \otimes V_{\alpha_{2}\alpha_{3}})).\]
This Massey product is not zero.
Hence the  DGA $\bigoplus_{\alpha} \bigwedge \g^{\ast}\otimes V_{\alpha}$ has a non-zero Massey product and it is not formal.
\begin{remark}
In \cite{Nar}, Narkawicz gave examples of complements $X$ of  hyperplane arrangements which are formal but for some diagonal representations of $\pi_{1}(X,x)$ the DGA $A^{\ast}(X,{\mathcal O}_{\rho})$ is non-formal.
\end{remark}
We have $d_{\alpha_{1}}(\zeta _{1}\otimes v_{\alpha_{1}})=0$ and the cohomology class $[ \zeta _{1}\otimes v_{\alpha_{1}}]\in
 H^{1}(\bigwedge \g_{\C}^{\ast} \otimes V_{\alpha_{1}} )$ is not zero.
We have 
\begin{multline*}
\omega^{3}=-6p(q+p) \alpha \wedge \beta\wedge \zeta _{1}\wedge \eta _{1}\wedge \zeta _{2}\wedge \eta _{2}\\
-6(p+q)q\alpha \wedge \beta\wedge \zeta _{2}\wedge \eta _{2}\wedge \zeta _{3}\wedge \eta _{3}\\
+6pq \alpha \wedge \beta\wedge \zeta _{1}\wedge \eta _{1}\wedge \zeta _{3}\wedge \eta _{3}\\
-6pq(p+q) \zeta _{1}\wedge \eta _{1}\wedge \zeta _{2}\wedge \eta _{2}
\wedge \zeta _{3}\wedge \eta _{3},
\end{multline*}
and 
\[\omega^{3}\wedge \zeta _{1}\otimes v_{\alpha_{1}}=-6(p+q)q\alpha \wedge \beta\wedge \zeta _{1}\wedge \zeta _{2}\wedge \eta _{2}\wedge \zeta _{3}\wedge \eta _{3}\otimes v_{\alpha_{1}}.
\]
Otherwise we have
\[ d_{\alpha_{1}}(\alpha \wedge \beta\wedge \zeta _{1}\wedge \eta _{1}\wedge \eta _{2}\wedge \eta _{3}\otimes v_{\alpha_{1}})
=-\alpha \wedge \beta\wedge \zeta _{1}\wedge \zeta _{2}\wedge \eta _{2}\wedge \zeta _{3}\wedge \eta _{3}\otimes v_{\alpha_{1}}.
\]
Hence $[\omega]^{3}\wedge([\zeta _{1}\otimes v_{\alpha_{1}}])=0$ and the operator $[\omega]^{3}\wedge$  is not injective.
\begin{theorem}
For $G/\Gamma$, the DGA $A^{\ast}(G/\Gamma,{\mathcal O}_{{\rm Ad}_{s\vert_{\Gamma}}}))$ is not formal and the linear operator
\[[\omega]^{3}\wedge:H^{1}(A^{\ast}(G/\Gamma,{\mathcal O}_{{\rm Ad}_{s\vert_{\Gamma}}}))\to H^{7}(A^{\ast}(G/\Gamma,{\mathcal O}_{{\rm Ad}_{s\vert_{\Gamma}}}))
\]
is not an isomorphism.
Thus ${\bf U}_{G}$ is not abelian.
In particular $G/\Gamma$ is not K\"ahler.
\end{theorem}
\end{example}

As above examples, comparing with untwisted versions, formality and the hard Lefschetz properties of the DGA $A^{\ast}(G/\Gamma,{\mathcal O}_{{\rm Ad}_{s\vert_{\Gamma}}})$ are useful criteria for formal and hard Lefschetz solvmanifolds to be not  K\"ahler.
But we have a non-K\"ahler symplectic solvmanifold such that $A^{\ast}(G/\Gamma,{\mathcal O}_{{\rm Ad}_{s\vert_{\Gamma}}})$ is formal and hard Lefschetz.
In \cite{Ara} Arapura showed that for a  simply connected solvable Lie group $G$ with a lattice $\Gamma$ if a solvmanifold $G/\Gamma$ admits a K\"ahler structure then $\Gamma$ is virtually abelian.
In \cite{Aus} it was proved that a lattice of a simply connected solvable Lie group $G$ is virtually nilpotent if and only if $G$ is type (I) i.e. for any $g\in G$ the all eigenvalues of ${\rm Ad}_{g}$ have absolute value 1.
Thus by Theorem \ref{SOLFORM} and \ref{SOLHARD},  we have:
\begin{corollary}
Let $G=\R^{n}\ltimes_{\phi} \R^{m}$ such that the action $\phi:\R^{n}\to {\rm Aut}(\R^{m})$ is semi-simple.
Suppose $G$ is not type (I) and has a lattice $\Gamma$.
Then $A^{\ast}(G/\Gamma,{\mathcal O}_{{\rm Ad}_{s\vert_{\Gamma}}})$ is formal but $G/\Gamma$  has no K\"ahler structure.
If $G/\Gamma$ has a symplectic form $\omega$, then the operator
\[[\omega]^{n-i}\wedge: H^{i}(A^{\ast}(G/\Gamma,{\mathcal O}_{{\rm Ad}_{s\vert_{\Gamma}}}))\to H^{2n-i}(A^{\ast}(G/\Gamma,{\mathcal O}_{{\rm Ad}_{s\vert_{\Gamma}}}))
\] 
is an isomorphism  for any $i\le n$ where $\dim G=2n$.
\end{corollary}

We give complex examples.

\begin{example}(\cite{Na})\\
Let $G=\C\ltimes_{\phi} \C^{2}$ with $\phi(x)= \left(
\begin{array}{ccc}
e^{x}&0\\
0&e^{-x}
\end{array}
\right)$.
Then $G$ has an invariant symplectic form.
 In \cite{Na}, it was shown that $G$ has a lattice $\Gamma$.
Thus $G/\Gamma$ is a non-K\"ahler complex solvmanifold but $A^{\ast}(G/\Gamma,{\mathcal O}_{{\rm Ad}_{s\vert_{\Gamma}}})$ is formal and hard Lefschetz.
\end{example}

\section{On isomorphism $H^{\ast}(G/\Gamma,\C)\cong H^{\ast}(\g_{\C})$} \label{APP}
Let $G$ be a simply connected solvable Lie group with a lattice $\Gamma$ and $\g$ be the Lie algebra of $G$.
We give new criteria for the isomorphism $H^{\ast}(G/\Gamma,\C)\cong H^{\ast}(\g_{\C})$ to hold by using Corollary \ref{GENN}.
Take a basis $X_{1},\dots,X_{n}$ of $\g_{\C}$ such that ${\rm Ad}_{s}$ is represented by diagonal matrices as ${\rm Ad}_{sg}=diag(\alpha_{1}(g),\dots,\alpha_{n}(g))$.
For $\{i_{1},\dots ,i_{p}\}\subset \{1,\dots ,n\}$ write $\alpha_{i_{1}\dots i_{p}}$ as the product of characters $\alpha_{i_{1}},\dots ,\alpha_{i_{p}}$.
\begin{corollary}\label{MOSEX}
Let $G$ be a simply connected solvable Lie group with a lattice $\Gamma$ and $\g$ be the Lie algebra of $G$. 
Suppose $(G,\Gamma)$ satisfies the following condition  :\\
 $(C_{G, \Gamma})$: For any $\{i_{1},\dots ,i_{p}\}\subset \{1,\dots ,n\}$ if the character $\alpha_{i_{1}\dots i_{p}}$ is non-trivial then the restriction of   $\alpha_{i_{1}\dots i_{p}\vert \Gamma}$ on $\Gamma$ is also non-trivial.\\
Then an isomorphism $H^{\ast}(G/\Gamma,\C)\cong H^{\ast}(\g_{\C})$ holds.
\end{corollary}
\begin{proof}
Let $x_{1},\dots, x_{n}$ be a basis of $ \g_{ \C}^{\ast}$ which is dual to $X_{1},\dots ,X_{n}$.
Consider the DGA $\left(\bigoplus_{\alpha\in A_{\Gamma}} \bigwedge \g^{\ast}_{\C}\otimes V_{\alpha}\right)^{\bf T}$ as Corollary \ref{GENN}.
By ${\rm Ad}_{sg}^{\ast}\cdot x_{i}=\alpha_{i}(g)^{-1}x_{i}$ we have 
\begin{multline*}
\left(\bigoplus_{\alpha\in A_{\Gamma}} \bigwedge \g^{\ast}_{\C}\otimes V_{\alpha}\right)^{\bf T}\\
=\left\langle x_{i_{1}}\wedge \dots \wedge x_{i_{p}}\otimes v_{ \alpha_{i_{1}\dots i_{p}}}{\Big \vert} \begin{array}{cc}1\le i_{1}<i_{2}<\dots <i_{p}\le n,\\  {\rm the \, \,  restriction \, \,  of \, \, \alpha_{i_{1}\dots i_{p}}\, \,  on  \, \, \Gamma \, \, is\, \,  trivial}\end{array}\right\rangle
\end{multline*}
as  the proof of Lemma \ref{uniso}.
Suppose $(G,\Gamma)$ satisfies the condition  $(C_{G, \Gamma})$.
Then we have
\[\left(\bigoplus_{\alpha\in A_{\Gamma}} \bigwedge \g^{\ast}_{\C}\otimes V_{\alpha}\right)^{\bf T}= \left(\bigwedge \g_{ \C}^{\ast}\right)^{\bf T}.
\]
Hence by Corollary \ref{GENN}, we have an isomorphism
\[H^{\ast}\left(  \left(\bigwedge \g_{ \C}^{\ast}\right)^{\bf T}\right)\cong H^{\ast}(G/\Gamma,\C).
\]
This implies that the inclusion $\bigwedge (\g_{\C})^{\ast} \subset A_{\C}^{\ast}(G/\Gamma)$ induces an isomorphism
\[H^{\ast}(\g_{ \C})\cong H^{\ast}(G/\Gamma,\C).\]
\end{proof}
\begin{remark}
We have examples such that  we can apply of this corollary but can not use Mostow's theorem(=Theorem \ref{Moss} ).
\end{remark}

\begin{example}
Let $G=\R\ltimes_{\phi}\R^{2}$ with $\phi(t)=\left(\begin{array}{cc} \cos \pi t & -\sin \pi t \\ \sin \pi t &\cos \pi t\end{array}\right)$.
Then $G$ has a lattice $\Gamma=\Z\ltimes \Z^{2}$.
In this case $G$ is not completely solvable and $(G,\Gamma)$ does not satisfies the Mostow's condition.
But diagonalization of ${\rm Ad}_{s}$ is given by ${\rm Ad}_{s}(t,x,y)=diag(1,e^{\pi  t\sqrt{-1}}, e^{-\pi t \sqrt{-1}})$ and hence  $(G,\Gamma)$ satisfies the condition  $(C_{, G, \Gamma})$.
Thus we have an isomorphism $H^{\ast}(\g_{\C})\cong H^{\ast}(G/\Gamma,\C)$.
\end{example}

For a character $\alpha$ of $G$, if the restriction of $\alpha$ on $\Gamma$ is trivial, then the image $\alpha(G) =\alpha(G/\Gamma)$ is compact and hence $\alpha$ is a unitary character.
Hence the above corollary reduce to the following corollary.
\begin{corollary}
Let $G$ be a simply connected solvable Lie group with a lattice $\Gamma$ and $\g$ be the Lie algebra of $G$. 
Suppose $G$ satisfies the following condition  :\\
 $(D_{ G})$:  For each $\{i_{1},\dots ,i_{p}\}\subset \{1,\dots ,n\}$  the character $\alpha_{i_{1}\dots i_{p}}$ is not a non-trivial unitary character.\\
Then an isomorphism $H^{\ast}(G/\Gamma,\C)\cong H^{\ast}(\g_{\C})$ holds.
\end{corollary}
Since the condition  $(D_{ G})$ does not concern with $\Gamma$, this corollary is more useful than the above corollary.
Clearly a completely solvable Lie group satisfies the condition $(D_{ G})$.
Hence this corollary is a generalization of Hattori's result in \cite{Hatt}.

\begin{example}
Let $G=\R^{s}\ltimes _{\phi}(\R^{s}\times \C) $ such that 
\[\phi(t_{1},\dots ,t_{s})=\left(
\begin{array}{ccccc}
e^{t_{1}}&  0&  \dots&0&0  \\
0&     \ddots &\ddots   &\vdots&\vdots  \\
\vdots& \ddots &e^{t_{s}}&0&0 \\
0&\dots&0& e^{-\frac{1}{2}(t_{1}+\dots t_{s})}\cos\varphi &-e^{-\frac{1}{2}(t_{1}+\dots t_{s})}\sin\varphi \\
0 &\dots &0& e^{-\frac{1}{2}(t_{1}+\dots t_{s})}\sin\varphi& e^{-\frac{1}{2}(t_{1}+\dots t_{s})}\cos\varphi  
\end{array}
\right),
\]
where $\varphi=c_{1}t_{1}+\dots +c_{s}t_{s}$.
Then a diagonalization of ${\rm Ad}_{s}$ is given by
\[{\rm Ad}_{s}=diag(e^{t_{1}},\dots, e^{t_{s}}, e^{-\frac{1}{2}(t_{1}+\dots t_{s})+\varphi\sqrt{-1}}, e^{-\frac{1}{2}(t_{1}+\dots t_{s})-\varphi\sqrt{-1}},1,\dots,1).
\]
By this, $G$ satisfies the condition $(D_{G})$ for any.
\begin{proposition}
For any lattice $\Gamma$, we have $b_{p}(G/\Gamma)=b_{2s+2-p}(G/\Gamma)={}_s{\rm C}_{p}$ for $1\le p\le s$ and $b_{s+1}(G/\Gamma)=0$.
\end{proposition}
\begin{proof}
For a coordinate $(t_{1},\dots t_{s},x_{1},\dots x_{s}, z)\in \R^{s}\ltimes _{\phi}(\R^{s}\times \C)$, the cochin complex $\bigwedge \g^{\ast}_{ \C}$  is generated by 
\[\{dt_{1},\dots ,dt_{s}, e^{-t_{1}}dx_{1},\dots, e^{-t_{s}}dx_{s}, e^{\frac{1}{2}(t_{1}+\dots t_{s})-\varphi\sqrt{-1}}dz, e^{\frac{1}{2}(t_{1}+\dots t_{s})+\varphi\sqrt{-1}}d{\bar z}\}.
\]
Since $G$ satisfies the condition $(D_{p, G})$, we have an isomorphism
\[ H^{p}(G/\Gamma,\C)\cong H^{\ast}\left(\left(\bigwedge^{p} \g_{ \C}^{\ast}\right)^{\bf T}\right).
\]
We have 
\[\left(\bigwedge^{p} \g_{ \C}^{\ast}\right)^{\bf T}=\langle dt_{i_{1}}\wedge\dots \wedge dt_{i_{p}}\vert 1\le i_{1}<\dots <i_{p}\le s\rangle\]
for $1\le p\le s$
and  $\left(\bigwedge^{s+1} \g_{ \C}^{\ast}\right)^{\bf T}=0$.
Since the restriction of the derivation on $\left(\bigwedge^{p} \g_{ \C}^{\ast}\right)^{\bf T}$ is $0$ for $1\le p\le s+1$, we have 
\[H^{\ast}\left(\left(\bigwedge^{p} \g_{ \C}^{\ast}\right)^{\bf T}\right)\cong \langle dt_{i_{1}}\wedge\dots \wedge dt_{i_{p}}\vert 1\le i_{1}<\dots <i_{p}\le s\rangle.\]
By the Poincar\'e duality, we have the proposition.
\end{proof}

We can construct a lattice of $G$ by using of number theory.
Let $K$ be a finite extension field of $\Q$ with the degree $s+2$($s>0$).
Suppose $K$ admits embeddings $\sigma_{1},\dots \sigma_{s},\sigma_{s+1}, \sigma_{s+2}$ into $\C$ such that $\sigma_{1},\dots ,\sigma_{s}$ are real embeddings and $\sigma_{s+1}, \sigma_{s+2}$ are complex ones satisfying $\sigma_{s+1}=\bar \sigma_{s+2}$. 
We can choose $K$ admitting such embeddings(see \cite{OT}).
Denote ${\mathcal O}_{K}$ the ring of algebraic integers of $K$, ${\mathcal O}_{K}^{\ast}$ the group of units in ${\mathcal O}_{K}$ and 
\[{\mathcal O}_{K}^{\ast\, +}=\{a\in {\mathcal O}_{K}^{\ast}: \sigma_{i}>0 \,\, {\rm for \,\,  all}\,\, 1\le i\le s\}.
\]  
Define $\sigma :{\mathcal O}_{K}\to \R^{s}\times \C$ by
\[\sigma(a)=(\sigma_{1}(a),\dots ,\sigma_{s}(a),\sigma_{s+1}(a))
\]
for $a\in {\mathcal O}_{K}$.
Then the image $\sigma({\mathcal O}_{K})$ is a lattice in $\R^{s}\times \C$.
We denote 
\begin{multline*}
\sigma(a)\cdot \sigma(b)\\
=(\sigma_{1}(a)\sigma_{1}(b),\dots ,\sigma_{s}(a)\sigma_{s}(b),\sigma_{s+1}(a)\sigma_{s+1}(b),\dots ,\sigma_{s+t}(a)\sigma_{s+t}(b))
\end{multline*}
for $a,b\in  {\mathcal O}_{K}$.
Define $l:{\mathcal O}_{K}^{\ast\, +}\to \R^{s+1}$ by 
\[l(a)=(\log \vert \sigma_{1}(a)\vert,\dots ,\log \vert \sigma_{s}(a)\vert , 2\log \vert \sigma_{s+1}(a)\vert)
\]
for $a\in {\mathcal O}_{K}^{\ast\, +}$.
Then by Dirichlet's units theorem, $l({\mathcal O}_{K}^{\ast\, +})$ is a lattice in the vector space $L=\{x\in \R^{s+1}\vert \sum_{i=1}^{s+1} x_{i}=0\}$.
By this we have a geometrical representation of the semi-direct product  $l({\mathcal O}_{K}^{\ast\, +})\ltimes_{\phi} \sigma({\mathcal O}_{K})$ with 
\[\phi(t_{1},\dots, t_{s+1})(\sigma(a))=\sigma(l^{-1}(t_{1},\dots, t_{s+1}))\cdot\sigma (a)
\]
for $(t_{1},\dots ,t_{s+1})\in l({\mathcal O}_{K}^{\ast\, +})$.
Since $l({\mathcal O}_{K}^{\ast\, +})$ and $\sigma({\mathcal O}_{K})$ are lattices of $L$ and $\R^{s}\times \C$ respectively, we have a extension $\bar \phi:L\to {\rm Aut} (\R^{s}\times \C) $ of $\phi$ and $l({\mathcal O}_{K}^{\ast\, +})\ltimes_{\phi} \sigma({\mathcal O}_{K})$ can be seen as a lattice of $L\ltimes_{\bar \phi} (\R^{s}\times \C)$.
Since we have $\phi(t_{1},\dots t_{s+1})=diag(e^{t_{1}},\dots, e^{t_{s}}, \sigma_{s+1}(l^{-1}(t_{1},\dots, t_{s+1})))$ and $\sigma_{s+1}$ is a complex embedding of $K$, 
for some $c_{1},\dots ,c_{s}\in \R$,  the Lie group $L\ltimes_{\bar \phi} (\R^{s}\times \C)$ is identified with the Lie group $G$ as above.
\begin{remark}
In \cite{OT}, for each $s$ Oeljeklaus and  Toma constructed a LCK (locally conformal K\"ahler) structure on the manifold $G/l({\mathcal O}_{K}^{\ast\, +})\ltimes_{\phi} \sigma({\mathcal O}_{K})$ and showed that for $s=2$ this LCK manifold is a counter example of Vaisman's conjecture(i.e. Every compact LCK manifold has odd odd Betti number).
By the above proposition, for $s=2m$ the Betti number $b_{p}=b_{4m+2-p}={}_{2m}C_{p}$ is even for odd number $1\le p < 2m$.
Hence for any even $s$, $G/l({\mathcal O}_{K}^{\ast\, +})\ltimes_{\phi} \sigma({\mathcal O}_{K})$ is also a counter example of Vaisman's conjecture.
\end{remark}
\end{example}

\ \\
{\bf  Acknowledgements.} 

The author would like to express his gratitude to   Toshitake Kohno for helpful suggestions and stimulating discussions.
He would also like to thank    Oliver Baues,  Keizo Hasegawa and Takumi Yamada for their active interests in this paper.
This research is supported by JSPS Research Fellowships for Young Scientists.
The author would like to express many thanks to the referee for his careful reading of the earlier version of manuscript with several important remarks, which lead to many improvements in the revised version.

\end{document}